\documentclass{amsart}
\usepackage{amssymb}
\usepackage{amsfonts}
\usepackage{amsmath}
\usepackage{mathrsfs}
\usepackage[numbers,sort&compress]{natbib}

\newtheorem{theorem}{Theorem}[section]
\newtheorem{lemma}[theorem]{Lemma}
\theoremstyle{definition}

\newtheorem{proposition}[theorem]{Proposition}

\newtheorem{corollary}[theorem]{Corollary}
\newtheorem{remark}[theorem]{Remark}
\numberwithin{equation}{section}

\begin{document}
\title[]{%
Direct and inverse problems for a third-order self-adjoint differential operator with non-local potential functions}
\author{Yixuan Liu and Mingming Zhang}

\begin{abstract}
The direct and inverse problems for a third-order self-adjoint differential operator with non-local potential functions are considered. Firstly, the  multiplicity for eigenvalues of the operator is analyzed, and it is proved that the differential operator has simple eigenvalues, except for finitely many eigenvalues of multiplicity two or three. Then the expressions of eigenfunctions and resolvent are obtained. Finally, the inverse problem for recovering non-local potential functions is solved.
$\mathbf{Keywords:}$ Direct problems; Inverse problems; Third-order self-adjoint differential operator; Non-local potential functions.
\end{abstract}

\maketitle



\section{Introduction}

The paper investigates the direct and inverse problems for the third-order self-adjoint differential operator $L_{\alpha}$ with non-local potential functions in $L_{\mathbb{C}}^2(0,1)$ defined by
\begin{eqnarray*}
&&\left(L_{\alpha}y\right)(x):=\mathrm{i}y^{\prime\prime\prime}(x)+\alpha\int_{0}^{1}y(t)\overline{v}(t)\mathrm{d}tv(x), \quad\alpha\in\mathbb{R}, \quad v\in L_{\mathbb{c}}^{2}(0,1),
\end{eqnarray*}
which domain $\mathcal{D}(L_{\alpha})$ is formed by the functions $y\in W_{3}^{2}(0,1)$ satisfying the boundary conditions
\begin{equation}
y^{\prime\prime}(0)=0, \quad y^{\prime}(1)=y^{\prime}(0), \quad y^{\prime\prime}(1)=0.\label{fgh}
\end{equation}
We use $L_{0}$ to denote the self-adjoint ordinary differential operator
\begin{equation*}
\left(L_{0}y\right)(x):=\mathrm{i}y^{\prime\prime\prime}(x) \label{L0y}
\end{equation*}
acting on the space $L^{2}_{\mathbb{C}}[0,1]$, which domain is
\begin{equation}
\mathcal{D}\left({L}_{0}\right)=\left\{y\in W_{3}^2(0,1)\left \vert
y^{\prime\prime}(0)=0, y^{\prime}(1)=y^{\prime}(0), y^{\prime\prime}(1)=0
\right.\right\}.\label{D0}
\end{equation}
And the operator $L_{\alpha}$ can be considered as a one-dimensional perturbation of the operator $L_{0}$.

In recent years, the study of third-order differential operators has attracted extensive attention [1--8]. The third-order differential operator emerges in the construction of the Lax pair, which is used to solving nonlinear Boussinesq and Degasperis--Procesi equations. These two nonlinear equations are used to study the long waves in shallow water and vibrations of a cubic string (see [9--11] and the references therein). In addition, non-local potential functions are widely applied in the field of mathematical physics [6--7, 12--15].

Zolotarev [6] and Liu [7] explored the spectrum of the third-order self-adjoint differential operator with non-local potential functions and coupled boundary conditions on a finite interval, proving that the spectrum of the differential operator consists of simple eigenvalues and finitely many eigenvalues of multiplicity $2$.
Moreover, for a given self-adjoint operator with simple discrete spectrum, Dobosevych [16] considered its rank-one perturbations, pointing out that the multiplicity for eigenvalues of the perturbed operator is at most $2$.
However, in this paper, the spectrum of the operator $L_{\alpha}$ is more complex, as it not only includes simple eigenvalues and finitely many eigenvalues of multiplicity $2$, but may also include an eigenvalue of multiplicity $3$.
The potential reason for the occurrence of an eigenvalue of multiplicity $3$ is that the operator $L_{0}$ has an eigenvalue of multiplicity $2$ (see Lemma \ref{asymptotic}).
Furthermore, the boundary conditions (\ref{fgh}) can be regarded as a special case of the boundary conditions
\begin{equation}
\label{beta}
\cos\gamma y(0)-\mathrm{i}\sin \gamma y^{\prime\prime}(0)=0, \quad y^{\prime}(1)=\mathrm{e}^{\mathrm{i}\phi}y^{\prime}(0), \quad \cos\beta y(1)-\mathrm{i}\sin \beta y^{\prime\prime}(1)=0 
\end{equation}
in the appendix under the condition of $\cos^2\gamma+\cos^2\beta+(\mathrm{e}^{\mathrm{i}\phi}-1)^2=0$ (the boundary conditions (\ref{beta}) are referred to as the general separated boundary conditions of the third-order self-adjoint operator in Zolotarev [8]).
If condition $\cos^2\gamma+\cos^2\beta+(\mathrm{e}^{\mathrm{i}\phi}-1)^2\neq0$ is met, the direct and inverse problems for the third-order differential operator with non-local potential and boundary conditions are similar to [6-7, 15].

The structure of this paper is as follows. 
Section $2$ gives some useful symbols and properties. 
In section $3$, we first derive the expression for the characteristic function of the operator $L_{0}$ with the boundary conditions (\ref{D0}), and then use this expression to conduct the asymptotic estimates for eigenvalues of operator $L_{0}$ (see Lemma \ref{asymptotic}). 
The eigenvalues, eigenfunctions, and resolvent of the operator ${L}_{\alpha}$ are investigated in Section $4$,
and we deduce that the spectrum of the differential operator ${L}_{\alpha}$ consists of simple eigenvalues and  finitely many eigenvalues of multiplicity two or three (see Lemma \ref{spectrum} and Remark \ref{finite}). 
In section $5$, we solve the inverse problem, and Theorem \ref{four spectra} indicates that the non-local potential function of the operator  ${L}_{\alpha}$ can be reconstructed by four spectra. Section $6$ discusses the multiplicity for eigenvalues of the third-order differential operator with the boundary conditions (\ref{beta}) .

\section{Preliminaries}
For $g,h\in L_{\mathbb{C}}^2(0,1)$, we define  $\left<g,h\right>=\int_{0}^{1}g(x)\overline{h}(x)\mathrm{d}x$ as
the inner product and $||g||=\left<g,g\right>^\frac{1}{2}$ as the norm.
\begin{proposition}
\label{2.1}
{\rm \cite[Proposition 2.1] {LY2024}}
Denote $\omega=\mathrm{e}^{\frac{2\pi}{3}\mathrm{i}}=-\frac{1}{2}+\frac{\sqrt 3}{2}\mathrm{i}$, the following equalities hold:
\begin{eqnarray*}
&&\omega^2=-\frac{1}{2}-\frac{\sqrt 3}{2}\mathrm{i}=\overline{\omega},\quad \omega^{3}=1, \quad 1+\omega+\omega^{2}=0,\\\notag
&&1-\omega=\mathrm{i}\sqrt 3 \omega^{2},\quad\omega-\omega^{2}=\mathrm{i}\sqrt 3,\quad\omega^{2}-1=\mathrm{i}\sqrt 3 \omega.\label{omega}
\end{eqnarray*}
\end{proposition}

Let
\begin{equation}
c(z)=\frac{1}{3}\sum_{n=1}^{3}\mathrm{e}^{\omega^{n}z},\quad
s(z)=\frac{1}{3}\sum_{n=1}^{3}\frac{1}{\omega^{n}}\mathrm{e}^{\omega^{n}z},\quad
d(z)=\frac{1}{3}\sum_{n=1}^{3}\omega^{n}\mathrm{e}^{\omega^{n}z} \label{csd}
\end{equation}
denote the fundamental solutions of equation $y^{\prime\prime\prime}=y$ determined by the initial conditions
\begin{equation*}
\left(
\begin{array}{ccc}
c(0) & s(0) & d(0) \\
c^{\prime }(0) & s^{\prime }(0) &d^{\prime }(0) \\
 c^{\prime \prime}(0) &s^{\prime \prime}(0) & d^{\prime \prime}(0)%
\end{array}%
\right) =I_{3}:=\left(
\begin{array}{ccc}
1 & 0 & 0 \\
0 & 1 & 0 \\
0 & 0 & 1%
\end{array}%
\right).
\end{equation*}

\begin{lemma}
{\rm \cite[lemma 1.1] {ZVA2021}}
\label{entire}
The functions $c(z)$, $s(z)$ and $d(z)$ in {\rm (\ref{csd})} are entire functions of exponential type, and satisfy the following relations:

\rm{(1)} $s^\prime(z)=c(z)$, $d^\prime(z)=s(z)$, $c^\prime(z)=d(z)$;

\rm{(2)} $\overline{c(z)}=c(\overline{z})$, $\overline{s(z)}=s(\overline{z})$, $\overline{d(z)}=d(\overline{z})$;

\rm{(3)} $c(\omega z)=c(z)$, $s(\omega z)=\omega s(z)$, $d(\omega z)=\omega^{2} d(z)$;

{\rm (4)} Euler's formula
\begin{equation*}
\mathrm{e}^{\omega^{n}z}=c(z)+\omega^{n}s(z)+\frac{1}{\omega^{n}}d(z),\quad n=1,2,3;
\end{equation*}

\rm{(5)} the main identity
\begin{equation*}
c^{3}(z)+s^{3}(z)+d^{3}(z)-3c(z)s(z)d(z)=1;
\end{equation*}

\rm{(6)} the summation formulas
\begin{eqnarray*}
&&c(z_{1}+z_{2})=c(z_{1})c(z_{2})+s(z_{1})d(z_{2})+d(z_{1})s(z_{2}),\\
&&s(z_{1}+z_{2})=c(z_{1})s(z_{2})+s(z_{1})c(z_{2})+d(z_{1})d(z_{2}),\\
&&d(z_{1}+z_{2})=c(z_{1})d(z_{2})+s(z_{1})s(z_{2})+d(z_{1})c(z_{2});
\end{eqnarray*}

\rm{(7)}
\begin{eqnarray*}
&&3c^{2}(z)=c(2z)+2c(-z),\\
&&3s^{2}(z)=d(2z)+2d(-z),\\
&&3d^{2}(z)=s(2z)+2s(-z);
\end{eqnarray*}

\rm{(8)}
\begin{eqnarray*}
&&c^{2}(z)-s(z)d(z)=c(-z),\\
&&d^{2}(z)-s(z)c(z)=s(-z),\\
&&s^{2}(z)-d(z)c(z)=d(-z);
\end{eqnarray*}

\rm{(9)} Taylor's formulas
\begin{eqnarray*}
&&c(z)=\sum_{n=0}^{\infty}\frac{z^{3n}}{(3n)!}=1+\frac{z^{3}}{3!}+\frac{z^{6}}{6!}+\cdots,\\
&&s(z)=\sum_{n=0}^{\infty}\frac{z^{3n+1}}{(3n+1)!}=z+\frac{z^{4}}{4!}+\frac{z^{7}}{7!}+\cdots,\\
&&d(z)=\sum_{n=0}^{\infty}\frac{z^{3n+2}}{(3n+2)!}={\frac{z^{2}}{2!}+\frac{z^{5}}{5!}+\frac{z^{8}}{8!}+\cdots}.
\end{eqnarray*}
\end{lemma}
\begin{proof}
From Proposition \ref{omega}, this lemma can be proved.
\end{proof}

Based on the expressions (\ref{csd}) of the functions $c(z)$, $s(z)$ and $d(z)$, we denote
\begin{equation}
y_{1}(x,\lambda)=c\left(\mathrm{i} \sqrt[3]{\lambda} x\right),\\
\quad y_{2}(x,\lambda)=\frac{1}{\mathrm{i} \sqrt[3]{\lambda}}s\left(\mathrm{i} \sqrt[3]{\lambda} x\right),\\
\quad y_{3}(x,\lambda)=\frac{1}{\left(\mathrm{i} \sqrt[3]{\lambda}\right)^2}d\left(\mathrm{i} \sqrt[3]{\lambda} x\right),\label{yul}
\end{equation}
and the Fourier transforms for $v\in L_{\mathbb{C}}^{2}(0,1)$ are
\begin{eqnarray}
&&\widetilde{y}_{1}(\lambda):=\left<v(x),y_{1}\left(x,\overline{\lambda}\right)\right>=\frac{1}{3}\sum_{n=1}^{3}\widetilde{\psi}_{n}(\lambda),\notag\\
&&\widetilde{y}_{2}(\lambda):=\left<v(x),y_{2}\left(x,\overline{\lambda}\right)\right>=\frac{1}{3\left(\mathrm{i} \sqrt[3]{\lambda}\right)}\sum_{n=1}^{3}\frac{1}{\omega^{n}}\widetilde{\psi}_{n}(\lambda),\label{v}\\
&&\widetilde{y}_{3}(\lambda):=\left<v(x),y_{3}\left(x,\overline{\lambda}\right)\right>=\frac{1}{3\left(\mathrm{i} \sqrt[3]{\lambda}\right)^2}\sum_{n=1}^{3}\omega^{n}\widetilde{\psi}_{n}(\lambda),\notag
\end{eqnarray}
besides,
\begin{equation*}
\widetilde{\psi}_{n}(\lambda):=\int_{0}^{1}\mathrm{e}^{-\mathrm{i}\omega^{n}\sqrt[3]{\lambda} x}v(x)\mathrm{d} x,\quad n=1,2,3.
\end{equation*}
Let
\begin{equation*}
f^{*}(\lambda)=\overline{f\left(\overline{\lambda}\right)}
\end{equation*}
denote the operation of involution, then
\begin{eqnarray}
&&\widetilde{y}_{1}^{*}(\lambda):=\left<y_{1}(x,\lambda),v(x)\right>=\frac{1}{3}\sum_{n=1}^{3}\widetilde{\psi}_{n}^{*}(\lambda),\notag\\
&&\widetilde{y}_{2}^{*}(\lambda):=\left<y_{2}(x,\lambda),v(x)\right>=\frac{1}{3\left(\mathrm{i} \sqrt[3]{\lambda}\right)}\sum_{n=1}^{3}\omega^{n}\widetilde{\psi}_{n}^{*}(\lambda),\label{vstar}\\
&&\widetilde{y}_{3}^{*}(\lambda):=\left<y_{3}(x,\lambda),v(x)\right>=\frac{1}{3\left(\mathrm{i} \sqrt[3]{\lambda}\right)^2}\sum_{n=1}^{3}\frac{1}{\omega^{n}}\widetilde{\psi}_{n}^{*}(\lambda),\notag
\end{eqnarray}
where,
\begin{equation*}
\widetilde{\psi}_{n}^{*}(\lambda):=\int_{0}^{1}\mathrm{e}^{\mathrm{i}\frac{1}{\omega^{n}}\sqrt[3]{\lambda} x}\overline{v}(x)\mathrm{d} x,\quad n=1,2,3.
\end{equation*}

\begin{remark}
For $n=1,2,3$,  $\widetilde{\psi}_{n}(\lambda)$ and $\widetilde{\psi}_{n}^{*}(\lambda)$ are entire functions of exponential type, and hence (\ref{v}), (\ref{vstar}) are entire functions of exponential type.
\end{remark}
\begin{lemma}
\label{wcsd}
For $v\in L_{\mathbb{C}}^{2}(0,1)$, the following equalities hold:
\begin{eqnarray*}
&&y_{1}(1,\lambda)\widetilde{y}_{3}(\lambda)+y_{2}(1,\lambda)\widetilde{y}_{2}(\lambda)+y_{3}(1,\lambda)\widetilde{y}_{1}(\lambda)
=\widetilde{\nu}_{3}(-\lambda),\\
&&y_{1}(1,\lambda)\widetilde{y}_{2}(\lambda)+y_{2}(1,\lambda)\widetilde{y}_{1}(\lambda)-{\mathrm{i}\lambda}y_{3}(1,\lambda)\widetilde{y}_{3}(\lambda)
=\widetilde{\nu}_{2}(-\lambda),\\
&&y_{1}(1,\lambda)\widetilde{y}_{1}(\lambda)-{\mathrm{i}\lambda}y_{2}(1,\lambda)\widetilde{y}_{3}(\lambda)-{\mathrm{i}\lambda}y_{3}(1,\lambda)\widetilde{y}_{2}(\lambda)
=\widetilde{\nu}_{1}(-\lambda),
\end{eqnarray*}
where
\begin{eqnarray*}
&&\widetilde{\nu}_{1}(\lambda)=\left<\nu(x),y_{1}\left(x,\overline{\lambda}\right)\right>, \\
&&\widetilde{\nu}_{2}(\lambda)=\left<\nu(x),y_{2}\left(x,\overline{\lambda}\right)\right>, \\
&&\widetilde{\nu}_{3}(\lambda)=\left<\nu(x),y_{3}\left(x,\overline{\lambda}\right)\right>
\end{eqnarray*}
are the Fourier transforms for $\nu(x):=v(1-x)$.
\end{lemma}
\begin{proof}
From the functions $y_{1}(x,\lambda)$, $y_{2}(x,\lambda)$, $y_{3}(x,\lambda)$ in (\ref{yul}), the definitions of $\widetilde{y}_{1}(\lambda)$, $\widetilde{y}_{2}(\lambda)$, $\widetilde{y}_{3}(\lambda)$ in (\ref{v}) and the relations (2), (6) in Lemma \ref{entire}, we see
\begin{eqnarray*}
y_{1}(1,\lambda)\widetilde{y}_{3}(\lambda)+y_{2}(1,\lambda)\widetilde{y}_{2}(\lambda)+y_{3}(1,\lambda)\widetilde{y}_{1}(\lambda)
\!\!&=&\!\!\int_{0}^{1}y_{3}(1-t,\lambda)v(t)\mathrm{d}t\\
\!\!&=&\!\!\int_{0}^{1}y_{3}(x,\lambda)v(1-x)\mathrm{d}x\\
\!\!&=&\!\!\widetilde{\nu}_{3}(-\lambda).
\end{eqnarray*}
We can prove the last two equations in the same way.
\end{proof}

Let
\begin{equation}
m(\lambda):=\left<\int_{0}^{x}y_{3}(x-t,\lambda)v(t)\mathrm{d}t,v(x)\right>=\frac{1}{3\left(\mathrm{i} \sqrt[3]{\lambda}\right)^2}\sum_{n=1}^{3}\omega^{n}\eta_{n}(\lambda),\label{mlambda}
\end{equation}
where
\begin{equation*}
\eta_{n}(\lambda):=\int_{0}^{1}\mathrm{d}x\int_{0}^{x}\mathrm{e}^{\mathrm{i}\omega^{n}\sqrt[3]{\lambda} (x-t)}\overline{v}(x){v}(t)\mathrm{d}t,\quad n=1,2,3
\end{equation*}
are the Fourier transforms of the convolution.

\begin{lemma}
\label{mmstar}
The function $m(\lambda)$ {\rm (\ref{mlambda})} is an entire function of exponential type and satisfies
\begin{equation}
m(\lambda)+m^{*}(\lambda)=\widetilde{y}_{3}(\lambda)\widetilde{y}_{1}^{*}(\lambda)+\widetilde{y}_{2}(\lambda)\widetilde{y}_{2}^{*}(\lambda)+\widetilde{y}_{1}(\lambda)\widetilde{y}_{3}^{*}(\lambda).
\end{equation}
\end{lemma}
\begin{proof}
Since $\eta_{n}(\lambda)$ is the Fourier transform of the convolution
\begin{eqnarray*}
\eta_{n}(\lambda)=\int_{0}^{1}\mathrm{d}x\int_{0}^{x}\mathrm{e}^{\mathrm{i}\omega^{n}\sqrt[3]{\lambda} (x-t)}\overline{v}(x){v}(t)\mathrm{d}t=\int_{0}^{1}\mathrm{e}^{\mathrm{i}\omega^{n}\sqrt[3]{\lambda} s}\mathrm{d}s\int_{s}^{1}\overline{v}(x){v}(x-s)\mathrm{d}x,
\end{eqnarray*}
then $\eta_{n}(\lambda)$ is an entire function of exponential type, thus $m(\lambda)$ is an entire function of exponential type.

Due to (\ref{mlambda}), one has
\begin{eqnarray*}
m(\lambda)+m^{*}(\lambda)=\frac{1}{3\left(\mathrm{i} \sqrt[3]{\lambda}\right)^2}[\omega(\eta_{1}(\lambda)+\eta_{2}^{*}(\lambda))+\omega^{2}(\eta_{2}(\lambda)+\eta_{1}^{*}(\lambda))+\omega^{3}(\eta_{3}(\lambda)+\eta_{3}^{*}(\lambda))].
\end{eqnarray*}
Using integration by parts, we get
\begin{eqnarray*}
\eta_{n}(\lambda)=\int_{0}^{1}\mathrm{e}^{-\mathrm{i}\omega^{n}\sqrt[3]{\lambda} t}{v}(t)\mathrm{d}t\int_{0}^{1}\mathrm{e}^{\mathrm{i}\omega^{n}\sqrt[3]{\lambda} x}\overline{v}(x)\mathrm{d}x-\int_{0}^{1}\mathrm{e}^{-\mathrm{i}\omega^{n}\sqrt[3]{\lambda} t}{v}(t)\mathrm{d}t\int_{0}^{t}\mathrm{e}^{\mathrm{i}\omega^{n}\sqrt[3]{\lambda} x}\overline{v}(x)\mathrm{d}x,
\end{eqnarray*}
therefore
\begin{eqnarray*}
&&\eta_{1}(\lambda)=\widetilde{y}_{1}(\lambda)\widetilde{y}_{2}^{*}(\lambda)-\eta_{2}^{*}(\lambda),\\
&&\eta_{2}(\lambda)=\widetilde{y}_{2}(\lambda)\widetilde{y}_{1}^{*}(\lambda)-\eta_{1}^{*}(\lambda),\\
&&\eta_{3}(\lambda)=\widetilde{y}_{3}(\lambda)\widetilde{y}_{3}^{*}(\lambda)-\eta_{3}^{*}(\lambda).
\end{eqnarray*}
And hence
\begin{eqnarray*}
m(\lambda)+m^{*}(\lambda)\!\!&=&\!\!\frac{1}{3\left(\mathrm{i} \sqrt[3]{\lambda}\right)^2}[\omega\widetilde{y}_{1}(\lambda)\widetilde{y}_{2}^{*}(\lambda)+\omega^{2}\widetilde{y}_{2}(\lambda)\widetilde{y}_{1}^{*}(\lambda)+\omega^{3}\widetilde{y}_{3}(\lambda)\widetilde{y}_{3}^{*}(\lambda)]\\
\!\!&=&\!\!\int_{0}^{1}\mathrm{d}x\int_{0}^{1}y_{3}(x-t,\lambda)\overline{v}(x){v}(t)\mathrm{d}t\\
\!\!&=&\!\!\int_{0}^{1}\mathrm{d}x\int_{0}^{1}\left[y_{1}(x,\lambda)y_{3}(-t,\lambda)+y_{2}(x,\lambda)y_{2}(-t,\lambda)\right.\\
\!\!&&\!\!\left.+y_{3}(x,\lambda)y_{1}(-t,\lambda)\right]\overline{v}(x){v}(t)\mathrm{d}t.
\end{eqnarray*}
According to (6) in Lemma \ref{entire}, (\ref{v}) and (\ref{vstar}), we obtain
\begin{eqnarray*}
m(\lambda)+m^{*}(\lambda)=\widetilde{y}_{3}(\lambda)\widetilde{y}_{1}^{*}(\lambda)+\widetilde{y}_{2}(\lambda)\widetilde{y}_{2}^{*}(\lambda)+\widetilde{y}_{1}(\lambda)\widetilde{y}_{3}^{*}(\lambda).
\end{eqnarray*}
\end{proof}

\section{The operator $L_{0}$}

In this section, we discuss the asymptotic estimates for eigenvalues, eigenfunctions and resolvent of the operator $L_{0}$.

From (\ref{csd}), we find $y_{1}(x,\lambda)$, $y_{2}(x,\lambda)$ and $y_{3}(x,\lambda)$ (\ref{yul}) are the solutions of
\begin{equation}
\mathrm{i}y^{\prime\prime\prime}(x)=\lambda y(x),\quad \lambda \in \mathbb{C} \label{L0}
\end{equation}
satisfying the initial conditions
\begin{equation*}
\left(
\begin{array}{ccc}
y_{1}(0,\lambda) & y_{2}(0,\lambda) & y_{3}(0,\lambda) \\
y_{1}^{\prime}(0,\lambda) & y_{2}^{\prime}(0,\lambda) & y_{3}^{\prime}(0,\lambda) \\
y_{1}^{\prime \prime}(0,\lambda) & y_{2}^{\prime \prime}(0,\lambda) & y_{3}^{\prime \prime}(0,\lambda)%
\end{array}%
\right) =I_{3}.
\end{equation*}
Then each unique solution of (\ref{L0}) with initial conditions
\begin{equation}
\left(y(0),y^{\prime}(0),y^{\prime\prime}(0)\right)=\left(c_{1},c_{2},c_{3}\right)\in \mathbb{C}^3 \label{initial}
\end{equation}
is
\begin{equation}
y_{0}(x,\lambda)=c_{1}y_{1}(x,\lambda)+c_{2}y_{2}(x,\lambda)+c_{3}y_{3}(x,\lambda). \label{solutionl0}
\end{equation}
So, according to the method of variation of constants, we find the unique solution to the third order inhomogeneous differential equation
\begin{equation}
\mathrm{i}y^{\prime\prime\prime}(x)=\lambda y(x)+f(x), \quad f\in L_{\mathbb{C}}^2(0,1)\label{cauchy}
\end{equation}
satisfying the initial conditions (\ref{initial}) is
\begin{equation}
y(x,\lambda)=c_{1}y_{1}(x,\lambda)+c_{2}y_{2}(x,\lambda)+c_{3}y_{3}(x,\lambda)-\mathrm{i}\int_{0}^{x}y_{3}(x-t,\lambda)f(t)\mathrm{d}t.\label{cauchys}
\end{equation}

\subsection{The properties for eigenvalues of the operator ${L}_{0}$}\

In this subsection, we deduce the characteristic function, the asymptotic estimates for eigenvalues and the eigenfunctions of the operator ${L}_{0}$.

\begin{lemma}
\label{l0s}
Let
\begin{equation*}
M(0,\lambda):=\left(
\begin{array}{ccc}
y_{1}^{\prime}(1,\lambda) & y_{2}^{\prime}(1,\lambda)-1  \\
y_{1}^{\prime \prime}(1,\lambda) & y_{2}^{\prime \prime}(1,\lambda)%
\end{array}%
\right).
\end{equation*}
The spectrum of the operator ${L}_{0}$ is denoted by
\begin{equation*}
\sigma\left({L}_{0}\right):=\{\lambda_{n} \vert n\in \mathbb{Z}\},
\end{equation*}
where $\lambda_{n}$ is the zero of the characteristic function 
\begin{eqnarray}
\Delta (0,\lambda):=\det M(0,\lambda)=-\mathrm{i}\lambda [y_{2}(-1,\lambda)+y_{2}(1,\lambda)]. \label{Delta0}
\end{eqnarray}
Besides, $\lambda_{0}=0$, $\lambda_{n}=-\lambda_{-n}$, $n\in \mathbb{N}^{*}$.
\end{lemma}
\begin{proof}
As the unique solution of (\ref{L0})-(\ref{initial}) is (\ref{solutionl0}), then through the boundary conditions (\ref{D0}), we obtain the system of linear equations for $c_1$ and $c_2$,
\begin{equation}
M(0,\lambda)\left(
\begin{array}{c}
c_{1} \\
c_{2}
\end{array}
\right)=\left(\begin{array}{c}
0 \\
0
\end{array}%
\right).\label{ma0a1a2}
\end{equation}
Obviously, $\lambda$ is an eigenvalue of the operator ${L}_{0}$ if and only if the system (\ref{ma0a1a2}) has non-trivial solution. Due to Proposition \ref{2.1}, and (1), (8) in Lemma \ref{entire}, we can calculate the characteristic function 
\begin{eqnarray*}
\Delta (0,\lambda)\!\!&=&\!\!-\mathrm{i}\lambda [y_{2}(-1,\lambda)+y_{2}(1,\lambda)]\\
\!\!&=&\!\!-\frac{2}{3}{\left(\sqrt[3]{\lambda}\right)}^{2}\left(\cos \sqrt[3]{\lambda}+\omega^{2}\cos\omega \sqrt[3]{\lambda}+\omega\cos\omega^{2}\sqrt[3]{\lambda}\right),
\end{eqnarray*} 
so one obtains $\lambda_{n}=-\lambda_{-n}$, $n\in \mathbb{N}^{*}$.
\end{proof}

\begin{lemma}
\label{asymptotic}
The characteristic function $\Delta (0,\lambda)$ is an entire function of exponential type. The real zeros of the function $\Delta (0,\lambda)$ are $\lambda_{n}(0)$ $(n\in\mathbb{Z})$, where $\lambda_{n}(0)$ are simple and enumerated in the ascending order, except for $\lambda_{0}(0)=0$ which has multiplicity $2$. Additionally, $\sqrt[3]{{\lambda}_{n}(0)}\in((2n-1)\pi, (2n+1)\pi)$ $(n\in \mathbb{Z})$, and the following asymptotic formula holds:
\begin{equation}
\sqrt[3]{{\lambda}_{n}(0)}\rightarrow2\pi n+\frac{\pi}{3}{\rm sgn}(n)+o\left(\frac{1}{n}\right)\quad (n\rightarrow\infty).\label{guji}
\end{equation}
\end{lemma}
\begin{proof}
The function $\Delta (0,\lambda)$ only has real zeros. If $\Delta (0,\lambda)$ has complex zeros, it contradicts the self-adjointness of the operator ${L}_{0}$.

(i) From the expression (\ref{Delta0}) and the identity
$y_{2}(-1,0)+y_{2}(1,0)=0$, we find ${\lambda}_{0}(0)=0$ is a zero of the function $\Delta (0,\lambda)$ which has multiplicity $2$.

(ii) For $\lambda\neq0$. The function (\ref{Delta0}) can be expressed as
\begin{eqnarray}
\Delta(0,\lambda)
=-\frac{2}{3}{k}^{2}\left(\cos k+\omega^{2}\cos\omega k+\omega\cos\omega^{2}k\right), \quad k=\sqrt[3]{\lambda}, \label{0k}
\end{eqnarray}
so the equation $\Delta (0,\lambda)=0$ is equivalent to the relation
\begin{equation*}
\cos k+\omega^{2}\cos\omega k+\omega\cos\omega^{2} k=0.
\end{equation*}
Considering Proposition \ref{omega}, we have
\begin{equation*}
\cos k+\left(-\frac{1}{2}-\frac{\sqrt 3}{2}\mathrm{i}\right)\cos\left(-\frac{1}{2}+\frac{\sqrt 3}{2}\mathrm{i}\right) k+\left(-\frac{1}{2}+\frac{\sqrt 3}{2}\mathrm{i}\right)\cos\left(-\frac{1}{2}-\frac{\sqrt 3}{2}\mathrm{i}\right) k=0,
\end{equation*}
and a simple calculation yields
\begin{equation}
\cos\frac{k}{2}\left(\cos\frac{k}{2}-\cosh\frac{\sqrt3 k}{2}\right)=\sin\frac{k}{2}\left(\sin\frac{k}{2}-\sqrt3\sinh\frac{\sqrt3 k}{2}\right).\label{cos}
\end{equation}
For $k\neq0$, from $\sin\frac{k}{2}-\sqrt3\sinh\frac{\sqrt3 k}{2}\neq0$, the equality (\ref{cos}) can be rewritten as
\begin{equation*}
f(k)=\tan\frac{k}{2},
\end{equation*}
where
\begin{equation*}
f(k):=\frac{\cos\frac{k}{2}-\cosh\frac{\sqrt3 k}{2}}{\sin\frac{k}{2}-\sqrt3\sinh\frac{\sqrt3 k}{2}}.\label{flambda}
\end{equation*}
The function $f(k)$ is odd, $f(0)=0$, and one has
\begin{equation*}
f^\prime(k)=2\frac{\cos\frac{k}{2}\cosh\frac{\sqrt3 k}{2}-1}{\left(\sin\frac{k}{2}-\sqrt3\sinh\frac{\sqrt3 k}{2}\right)^2}.
\end{equation*}
Let $g(k)=\cos\frac{k}{2}\cosh\frac{\sqrt3 k}{2}-1$, then we find
\begin{equation*}
g^\prime(k)=-\frac{1}{2}\sin\frac{k}{2}\cosh\frac{\sqrt3 k}{2}+\frac{\sqrt3}{2}\cos\frac{k}{2}\sinh\frac{\sqrt3 k}{2}.
\end{equation*}
For $k<0$, due to $\cos\frac{k}{2}<\cosh\frac{\sqrt3 k}{2}$ and $\sinh\frac{\sqrt3k}{2}<\sin\frac{k}{2}$, it yields
\begin{eqnarray*}
g^\prime(k)\!\!&<&\!\!\frac{\sqrt3}{2}\cosh\frac{\sqrt3 k}{2}\sinh\frac{\sqrt3 k}{2}-\frac{1}{2}\sin\frac{k}{2}\cosh\frac{\sqrt3 k}{2}\\
\!\!&=&\!\!\frac{1}{2}\cosh\frac{\sqrt3 k}{2}\left(\sqrt3\sinh\frac{\sqrt3 k}{2}-\sin\frac{k}{2}\right)\\
\!\!&<&\!\!0.
\end{eqnarray*}
Hence for $k\in\mathbb{R}_{-}$, we find $g(k)>g(0)=0$, and $f^\prime(k)>0$.
Since $\displaystyle\lim_{k\rightarrow -\infty}f(k)=-\frac{1}{\sqrt3}$, we can get $f(k)$ monotonically increases from $f(k)\rightarrow-\frac{1}{\sqrt3}$ $(k\rightarrow-\infty)$ to $f(0)=0$. According to its oddness, the function $f(k)$ ($k\in\mathbb{R}_{+}$) also monotonically increases from $f(0)=0$ to $f(k)\rightarrow\frac{1}{\sqrt3}$ $(k\rightarrow+\infty)$.

Therefore, in each interval $((2n-1)\pi, (2n+1)\pi)$ $(n\in \mathbb{Z}\setminus\{0\})$, the equation $f(k)=\tan\frac{k}{2}$ has only one root $k$.

From $f(k)=\tan\frac{k}{2}\rightarrow\pm\frac{1}{\sqrt3} $$(k\rightarrow\pm\infty)$ and $k=\sqrt[3]{\lambda}$, one obtains the asymptotic formula (\ref{guji}).
\end{proof}

\begin{lemma}
{\rm (1)} The eigenfunctions of the operator ${L}_{0}$ corresponding to $\lambda_{0}=0$ are
\begin{equation*}
u_{0}^{\mathrm{1}}(0,x)=\frac{e_{1}(x)}{||e_{1}(x)||}, \quad u_{0}^{\mathrm{2}}(0,x)=\frac{e_{2}(x)}{||e_{2}(x)||},
\end{equation*}
where $e_{1}(x)=1$, $e_{2}(x)=2\mathrm{i}x-\mathrm{i}$.\\
{\rm (2)} To each zero $\lambda_{n}$ of the characteristic function $\Delta(0,\lambda)$ $(\ref{Delta0})$, there corresponds an eigenfunction of the operator ${L}_{0}$,
\begin{equation*}
u_{ n}^{\mathrm{1}}(0,x)=\frac{y_{0}(x,\lambda_{n})}{||y_{0}(x,\lambda_{n})||}, \quad n\in \mathbb{Z}\setminus\{0\},
\end{equation*}
where $y_{0}(x,\lambda_{n})=y_{3}(1,\lambda_{n})y_{1}(x,\lambda_{n})-y_{2}(1,\lambda_{n})y_{2}(x,\lambda_{n})$.
\end{lemma}
\begin{proof}
{\rm (i)} For $\lambda_{0}=0$, we find that the two linearly independent solutions of equation $\mathrm{i}y^{\prime\prime\prime}(x)=0$ satisfying the boundary conditions (\ref{D0}) are $e_{1}(x)=1$, $e_{2}(x)=2\mathrm{i}x-\mathrm{i}$, thus we construct two orthonormal eigenfunctions
\begin{equation*}
u_{0}^{\mathrm{1}}(0,x)=\frac{e_{1}(x)}{||e_{1}(x)||}, \quad u_{0}^{\mathrm{2}}(0,x)=\frac{e_{2}(x)}{||e_{2}(x)||}.
\end{equation*}

{\rm (ii)} For $\lambda_{ n}$, $n\in \mathbb{Z}\setminus\{0\}$, from the proof of Lemma \ref{l0s}, the non-trivial solution of equation $\mathrm{i}y^{\prime\prime\prime}(x)=\lambda_{n}y(x)$ determined by the boundary conditions (\ref{D0}) is
\begin{equation*}
y_{0}(x,\lambda_{n})=y_{3}(1,\lambda_{n})y_{1}(x,\lambda_{n})-y_{2}(1,\lambda_{n})y_{2}(x,\lambda_{n}).
\end{equation*}
Hence the normalized eigenfunction with respect to $\lambda_{n}$, $n\in\mathbb{Z}\setminus\{0\}$ is $u_{ n}^{\mathrm{1}}(0,x)=\frac{y_{0}(x,\lambda_{n})}{||y_{0}(x,\lambda_{n})||}$.
\end{proof}

\begin{remark}
The eigenfunctions $u_{n}^{i}:=u_{n}^{i}(0,x)$ of the operator ${L}_{0}$ satisfy
\begin{equation*}
\left<u_{n}^{i},u_{m}^{i}\right>=\delta_{n,m}:=\left\{\begin{array}{ll}
1,&m=n,\quad i=j,\quad m,n\in \mathbb{Z},\\
0,&m\neq n, \quad m,n\in \mathbb{Z}\setminus\{0\},\\
0,&m=n=0,\quad i\neq j,
\end{array}
\right.
\end{equation*}
and form an orthonormal basis in the Hilbert space $L_{\mathbb{C}}^{2}(0,1)$.
\end{remark}

\subsection{The resolvent of the operator ${L}_{0}$}\

In this subsection, we derive the expression for the resolvent of the operator ${L}_{0}$.
\begin{proposition}
{\rm \cite{BS2005}}
The resolvent $R_{{L}_{0}}(\lambda)=\left({L}_{0}-\lambda I\right)^{-1}$ of the operator ${L}_{0}$ is
\begin{equation*}
R_{{L}_{0}}(\lambda)f=\sum_{n\in \mathbb{Z}}\frac{f_{n}}{\lambda_{n}-\lambda},\label{rl0}
\end{equation*}
where
\begin{equation*}
f_{n}:=\left\{\begin{array}{ll}
f_{0}^{\mathrm{1}}u_{0}^{\mathrm{1}}+f_{0}^{\mathrm{2}}u_{0}^{\mathrm{2}}, &  {for\quad n=0},\\
f_{n}^{\mathrm{1}}u_{n}^{\mathrm{1}},&  {for\quad n\in \mathbb{Z}\setminus\{0\}},
\end{array}
\right.
\end{equation*}
and $f_{n}^{i}:=\left<f,u_{n}^{i}\right>$ are Fourier coefficients of $f\in L_{\mathbb{C}}^{2}(0,1)$ in the basis $\{u_{n}^{i}\}$.
\end{proposition}

\begin{lemma}
\label{rel0}
For $\lambda \in \mathbb{C}$, $f\in L^2_{\mathbb{C}}(0,1)$, the resolvent $R_{{L}_{0}}$ of the operator ${L}_{0}$ can be expressed by
\begin{eqnarray*}
\left(R_{{L}_{0}}\left(\lambda\right)f\right)(x)\!\!&=&\!\!\frac{\mathrm{i}}{\Delta (0,\lambda)}\int_{0}^{x}\left[-y_{1}(x,\lambda)y_{1}(-t,\lambda)y_{1}(-1,\lambda)+\mathrm{i}\lambda y_{2}(x,\lambda)y_{1}(-t,\lambda)\right.\\
\!\!&&\!\!\left.y_{3}(-1,\lambda)+\mathrm{i}\lambda y_{3}(x,\lambda)y_{1}(-t,\lambda)y_{2}(-1,\lambda)+\mathrm{i}\lambda y_{3}(x-t,\lambda)\right.\\
\!\!&&\!\!\left.y_{2}(1,\lambda)+y_{1}(x,\lambda)y_{1}(1-t,\lambda)\right.]f(t)\mathrm{d}t\\
\!\!&&\!\!+\frac{\mathrm{i}}{\Delta (0,\lambda)}\int_{x}^{1}\left[-y_{1}(x,\lambda)y_{1}(-t,\lambda)y_{1}(-1,\lambda)-\mathrm{i}\lambda y_{1}(x,\lambda)y_{3}(-t,\lambda)\right.\\
\!\!&&\!\!\left.y_{2}(-1,\lambda)+\mathrm{i}\lambda y_{2}(x,\lambda)y_{1}(-t,\lambda)y_{3}(-1,\lambda)+y_{1}(x,\lambda)y_{1}(1-t,\lambda)\right.\\
\!\!&&\!\!\left.-\mathrm{i}\lambda y_{2}(x,\lambda)y_{2}(-t,\lambda)y_{2}(-1,\lambda) \right.]f(t)\mathrm{d}t.
\end{eqnarray*}
\end{lemma}
\begin{proof}
Let $R_{{L}_{0}}(\lambda)f=\left({L}_{0}-\lambda I\right)^{-1}f=y$, then
\begin{equation}
{L}_{0}y=\lambda y+f,\label{f}
\end{equation}
and $(\ref{cauchys})$ is the solution of $(\ref{f})$ satisfying the initial conditions (\ref{initial}). Considering the boundary conditions (\ref{D0}) and $y(x,\lambda)$ in (\ref{cauchys}), we obtain the system of linear equations
\begin{equation*}
M(0,\lambda)\left(
\begin{array}{c}
c_{1}\\
c_{2}
\end{array}
\right)=\left(\begin{array}{c}
\mathrm{i}\displaystyle\int_{0}^{1}y_{2}(1-t,\lambda)f(t)\mathrm{d}t \\
\mathrm{i}\displaystyle\int_{0}^{1}y_{1}(1-t,\lambda)f(t)\mathrm{d}t
\end{array}%
\right)
\end{equation*}
relative to $c_{1}$, $c_{2}$, where Lemma \ref{l0s} gives $M(0,\lambda)$. 
Through the relations (1), (6), (8) in Lemma \ref{entire}, a straightforward calculation shows
\begin{eqnarray*}
c_{1}\!\!&=&\!\!\frac{\mathrm{i}}{\Delta(0,\lambda)}\int_{0}^{1}\left.[-\mathrm{i}\lambda y_{2}(1-t,\lambda)y_{3}(1,\lambda)-y_{1}(1-t,\lambda)y_{1}(1,\lambda)\right.\\
\!\!&&\!\!\left.+y_{1}(1-t,\lambda)\right.]f(t)\mathrm{d}t\\
\!\!&=&\!\!\frac{\mathrm{i}}{\Delta(0,\lambda)}\int_{0}^{1}\left.[-\mathrm{i}\lambda y_{3}(-t,\lambda)y_{2}(-1,\lambda)-y_{1}(-t,\lambda)y_{1}(-1,\lambda)\right.\\
\!\!&&\!\!\left.+y_{1}(1-t,\lambda)\right.]f(t)\mathrm{d}t,\\
c_{2}\!\!&=&\!\!\frac{\mathrm{i}}{\Delta(0,\lambda)}\int_{0}^{1}(-\mathrm{i}\lambda)\left.[y_{1}(1-t,\lambda)y_{3}(1,\lambda)-y_{2}(1-t,\lambda)y_{2}(1,\lambda)\right.]f(t)\mathrm{d}t\\
\!\!&=&\!\!\frac{\mathrm{i}}{\Delta(0,\lambda)}\int_{0}^{1}(-\mathrm{i}\lambda)\left.[y_{2}(-t,\lambda)y_{2}(-1,\lambda)-y_{1}(-t,\lambda)y_{3}(-1,\lambda)\right.]f(t)\mathrm{d}t.
\end{eqnarray*}
So, we find
\begin{eqnarray*}
y(x,\lambda)
\!\!&=&\!\!c_{1}y_{1}(x,\lambda)+c_{2}y_{2}(x,\lambda)-\mathrm{i}\int_{0}^{x}y_{3}(x-t,\lambda)f(t)\mathrm{d}t\\
\!\!&=&\!\!\frac{\mathrm{i}}{\Delta (0,\lambda)}\int_{0}^{1}\left[-y_{1}(x,\lambda)y_{1}(-t,\lambda)y_{1}(-1,\lambda)-\mathrm{i}\lambda y_{1}(x,\lambda)y_{3}(-t,\lambda)\right.\\
\!\!&&\!\!\left.y_{2}(-1,\lambda)+\mathrm{i}\lambda y_{2}(x,\lambda)y_{1}(-t,\lambda)y_{3}(-1,\lambda)-\mathrm{i}\lambda y_{2}(x,\lambda)y_{2}(-t,\lambda)\right.\\
\!\!&&\!\!\left.y_{2}(-1,\lambda)+y_{1}(x,\lambda)y_{1}(1-t,\lambda)\right]f(t)\mathrm{d}t-\mathrm{i}\int_{0}^{x}y_{3}(x-t,\lambda)f(t)\mathrm{d}t.
\end{eqnarray*}
From the equality (\ref{Delta0}), we can prove this lemma.
\end{proof}

\section{The operator ${L}_{\alpha}$}
In this section, we study the characteristic function, spectrum and eigenfunctions of the operator ${L}_{\alpha}$. Without loss of generality, assume that $\Vert v\Vert_{L^2_{\mathbb{C}}(0,1)}=1$. Domains of the operators ${L}_{\alpha}$ and ${L}_{0}$ coincide, $\mathcal{D}\left({L}_{\alpha}\right)=\mathcal{D}\left({L}_{0}\right)$.

\subsection{The characteristic function of the operator ${L}_{\alpha}$}\

According to (\ref{cauchy})-(\ref{cauchys}), the unique solution of ${L}_{\alpha}y=\lambda y$ satisfying the initial conditions (\ref{initial}) is
\begin{equation}
y(x,\lambda)=c_{1}y_{1}(x,\lambda)+c_{2}y_{2}(x,\lambda)+c_{3}y_{3}(x,\lambda)+\mathrm{i}\alpha\left<y,v\right>\int_{0}^{x}y_{3}(x-t,\lambda)v(t)\mathrm{d}t.\label{cauchysv}
\end{equation}
In this subsection, we explore the expression for the characteristic function of the operator ${L}_{\alpha}$ via (\ref{cauchysv}).
\begin{lemma}
Let
\begin{equation*}
M(\alpha,\lambda):=\left(
\begin{array}{cccc}
\widetilde{y}_{1}^{*}(\lambda) & \widetilde{y}_{2}^{*}(\lambda) & \widetilde{y}_{3}^{*}(\lambda) & \mathrm{i}\alpha m(\lambda)-1 \\
0 & 0 & 1 & 0 \\
-\mathrm{i}\lambda y_{3}(1,\lambda) & y_{1}(1,\lambda)-1 & y_{2}(1,\lambda) &\mathrm{i}\alpha\displaystyle\int_{0}^{1}y_{2}(1-t,\lambda)v(t)\mathrm{d}t\\
-\mathrm{i}\lambda y_{2}(1,\lambda) & -\mathrm{i}\lambda y_{3}(1,\lambda) & y_{1}(1,\lambda) & \mathrm{i}\alpha\displaystyle\int_{0}^{1}y_{1}(1-t,\lambda)v(t)\mathrm{d}t%
\end{array}%
\right).
\end{equation*}
The characteristic function of the operator ${L}_{\alpha}$ is
\begin{eqnarray}
\Delta(\alpha,\lambda):=\det M(\alpha, \lambda)=\Delta(0, \lambda)+\mathrm{i}\alpha[{F}(\lambda)-{F}^{*}(\lambda)],\label{deltaa}
\end{eqnarray}
where
\begin{eqnarray}
{F}(\lambda):=\widetilde{y}_{1}^{*}(\lambda)\widetilde{\nu}_{1}(-\lambda)+\mathrm{i}\lambda y_{2}(1,\lambda)m(\lambda). \label{lanm}
\end{eqnarray}
\end{lemma}

\begin{proof}
Multiplying equality (\ref{cauchysv}) by $\overline{v}(x)$ and integrating it from $0$
to $1$, and using $\widetilde{y}_{1}^{*}(\lambda)$, $\widetilde{y}_{2}^{*}(\lambda)$, $\widetilde{y}_{3}^{*}(\lambda)$ in (\ref{vstar}) and $m(\lambda)$ in (\ref{mlambda}),
we have
\begin{equation*}
c_{1}\widetilde{y}_{1}^{*}(\lambda)+c_{2}\widetilde{y}_{2}^{*}(\lambda)
+c_{3}\widetilde{y}_{3}^{*}(\lambda)+\left<y,v\right>(\mathrm{i}\alpha m(\lambda)-1)=0.\label{yv}
\end{equation*}
This together with the boundary conditions (\ref{fgh}), we get $c_{1}$, $c_{2}$, $c_{3}$ and $\left<y,v\right>$ satisfying the following system of linear equations,
\begin{equation}
M(\alpha, \lambda)
\left(\begin{array}{c}
c_{1}\\
c_{2}\\
c_{3}\\
\left<y,v\right>
\end{array}
\right)=\left(\begin{array}{c}
0\\
0\\
0\\
0
\end{array}
\right).\label{123yv}
\end{equation}
It is obvious that $\lambda$ is an eigenvalue of the operator ${L}_{\alpha}$ if and only if the system (\ref{123yv}) has non-trivial solution. Thereby, $\Delta(\alpha,\lambda)$ is the characteristic function of the operator ${L}_{\alpha}$. Owing to Lemma \ref{wcsd} and Lemma \ref{mmstar},  the calculation shows

\begin{eqnarray*}
\Delta(\alpha,\lambda)\!\!&=&\!\!\mathrm{i}\alpha\int_{0}^{1}y_{2}(1-t,\lambda)v(t)\mathrm{d}t
\left[-\mathrm{i}\lambda \widetilde{y}_{1}^{*}(\lambda)y_{3}(1,\lambda)+\mathrm{i}\lambda \widetilde{y}_{2}^{*}(\lambda)y_{2}(1,\lambda)\right]\\
\!\!&&\!\!\left.-\mathrm{i}\alpha\int_{0}^{1}y_{1}(1-t,\lambda)v(t)\mathrm{d}t\left[\widetilde{y}_{1}^{*}(\lambda) (y_{1}(1,\lambda)-1)+\mathrm{i}\lambda \widetilde{y}_{2}^{*}(\lambda)y_{3}(1,\lambda)\right]\right.\\
\!\!&&\!\!\left.-(\mathrm{i}\alpha m(\lambda)-1)\Delta(0,\lambda)\right.\\
\!\!&=&\!\!\mathrm{i}\alpha\left[\mathrm{i}\lambda y_{2}(1,\lambda)m(\lambda)-\mathrm{i}\lambda y_{2}(-1,\lambda)m^{*}(\lambda)-\widetilde{y}_{1}(\lambda)\widetilde{\nu}_{1}^{*}(-\lambda)+\widetilde{y}_{1}^{*}(\lambda)\widetilde{\nu}_{1}(-\lambda)\right]\\
\!\!&&\!\!\left.+\Delta(0,\lambda).\right.
\end{eqnarray*}
Using ${F}(\lambda)$ in (\ref{lanm}), we obtain (\ref{deltaa}). 
\end{proof}

\subsection{The resolvent and spectrum of the operator ${L}_{\alpha}$}\

In this subsection,  we discuss the resolvent and spectrum of the operator ${L}_{\alpha}$.

\begin{lemma}
\label{RLA}
The resolvent $R_{{L}_{\alpha}}(\lambda)=\left({L}_{\alpha}-\lambda I\right)^{-1}$ of the operator ${L}_{\alpha}$ is expressed by the resolvent $R_{{L}_{0}}(\lambda)=\left({L}_{0}-\lambda I\right)^{-1}$ of the operator ${L}_{0}$ and satisfies
\begin{eqnarray}
R_{{L}_{\alpha}}(\lambda)f\!\!&=&\!\!R_{{L}_{0}}(\lambda)f-\alpha\frac{\left<R_{{L}_{\alpha}}(\lambda)f,v\right>}{1+\alpha\left<R_{{L}_{0}}(\lambda)v,v\right>}\cdot R_{{L}_{0}}(\lambda)v\notag\\
\!\!&=&\!\!\sum_{n\in\mathbb{Z}}\frac{f_{n}}{\lambda_{n}-\lambda}-\alpha\frac{\displaystyle\sum_{n\in\mathbb{Z}}\frac{f_{n}\overline{v}_{n}}{\lambda_{n}-\lambda}}{1+\alpha\displaystyle\sum_{n\in\mathbb{Z}}\frac{\vert v_{n}\vert^{2}}{\lambda_{n}-\lambda}}\cdot \sum_{n\in\mathbb{Z}}\frac{v_{n}}{\lambda_{n}-\lambda},\label{rla}
\end{eqnarray}
where
\begin{equation*}
v_{n}:=\left\{\begin{array}{ll}
v_{0}^{\mathrm{1}}u_{0}^{\mathrm{1}}+v_{0}^{\mathrm{2}}u_{0}^{\mathrm{2}}, & {for\quad n=0},\\
v_{n}^{\mathrm{1}}u_{n}^{\mathrm{1}},&  {for\quad n\in \mathbb{Z}\setminus\{0\}},
\end{array}
\right.
\end{equation*}
and $v_{n}^{{i}}:=\left<v,u_{n}^{{i}}\right>$ are Fourier coefficients of $v\in L_{\mathbb{C}}^{2}(0,1)$ in the basis $\{u_{n}^{{i}}\}$.
\end{lemma}
\begin{proof}
The proof is similar to that of {\rm\cite[lemma 4.1] {LY2024}}.
\end{proof}

\begin{corollary}
\label{residue0}
The residue of the resolvent $R_{{L}_{\alpha}}(\lambda)f$ at the point $\lambda_{0}=0$ is
\begin{equation*}
\text{Res}_{0}R_{{L}_{\alpha}}(\lambda)f=f_{0}^{\mathrm{1}}u_{0}^{\mathrm{1}}+f_{0}^{\mathrm{2}}u_{0}^{\mathrm{2}}
-\frac{f_{0}^{\mathrm{1}}\overline{{v}_{0}^{\mathrm{1}}}+f_{0}^{\mathrm{2}}\overline{{v}_{0}^{\mathrm{2}}}}{v_{0}^{\mathrm{1}}\overline{{v}_{0}^{\mathrm{1}}}+v_{0}^{\mathrm{2}}\overline{{v}_{0}^{\mathrm{2}}}}(v_{0}^{\mathrm{1}}u_{0}^{\mathrm{1}}+v_{0}^{\mathrm{2}}u_{0}^{\mathrm{2}})=f_{0}.
\end{equation*}

(1) If $v_{0}=v_{0}^{\mathrm{1}}u_{0}^{\mathrm{1}}+v_{0}^{\mathrm{2}}u_{0}^{\mathrm{2}}=0$, i.e., $v_{0}^{\mathrm{1}}=v_{0}^{\mathrm{2}}=0$, we get
\begin{equation*}
\text{Res}_{0}R_{{L}_{\alpha}}(\lambda)f=f_{0}^{\mathrm{1}}u_{0}^{\mathrm{1}}+f_{0}^{\mathrm{2}}u_{0}^{\mathrm{2}}\neq0,
\end{equation*}
then $0$ is the eigenvalue of the operator ${L}_{\alpha}$ which has multiplicity $2$.

(2) For $v_{0}=v_{0}^{\mathrm{1}}u_{0}^{\mathrm{1}}+v_{0}^{\mathrm{2}}u_{0}^{\mathrm{2}}\neq 0$, then $\left<f_{0},v_{0}\right>=0$. Let $G_{0}=\{u_{0}^{\mathrm{1}}, u_{0}^{\mathrm{2}}\}$, the vectors $f_{0}$ belong to $G_{0}$ and are orthogonal to $v_{0}$, and thus the totality of $f_{0}$ forms the subspace $G_{0}^{1}$ in $G_{0}$ of dimension $1$, therefore $0$ is a simple eigenvalue of the operator ${L}_{\alpha}$.
\end{corollary}

\begin{remark}
For $v\in L_{\mathbb{C}}^2(0,1)$, if $v_{n}\neq0$, $n\in \mathbb{Z}\setminus\{0\}$, then $\lambda_{n}\in\sigma\left({L}_{0}\right)$ is not the eigenvalue of the operator ${L}_{\alpha}$.
\end{remark}
\begin{proof}
The deduce is similar to that of {\rm\cite[lemma 4.3] {LY2024}}.
\end{proof}

Naturally, we can divide the set $\sigma\left({L}_{0}\right)$ into two disjoint subsets $\sigma_{0}$ and $\sigma_{1}$, i.e., $\sigma\left({L}_{0}\right)=\sigma_{0}\cup\sigma_{1}$, where
\begin{equation}
\sigma_{0}:=\left\{\lambda_{n}^{0}=\lambda_{n}\in \sigma\left({L}_{0}\right)\vert v_{n}=0\right\},\quad\sigma_{1}:=\left\{\lambda_{n}^{1}=\lambda_{n}\in \sigma\left({L}_{0}\right)\vert v_{n}\neq0\right\}. \label{01}\\
\end{equation}
Formula (\ref{rla}) implies that zeros of the function
\begin{equation}
Q(\lambda):=1+\alpha\sum_{n\in\mathbb{Z}}\frac{\vert v_{n}\vert^{2}}{\lambda_{n}-\lambda}\label{qz}
\end{equation}
belong to the spectrum $\sigma\left({L}_{\alpha}\right)$ of the operator ${L}_{\alpha}$. Besides, the zeros of $Q(z)$ are real, simple and alternate with numbers $\lambda_{n}^{1}\in \sigma_{1}$ (Its proof process is similar to that of {\rm\cite[Proposition 4.4] {LY2024}}).
We denote
\begin{equation}
\sigma_{2}:=\left\{\mu_{n}\left\vert Q(\mu_{n})=0,  n\in\mathbb{Z}\right.\right\}.\label{02}
\end{equation}
A number $\mu_{n}$ may coincide with $\lambda_{n}^{0}$ from $\sigma_{0}$, i.e., $\sigma_{0}\cap\sigma_{2}\neq\emptyset$, and in this case  
the multiplicity of $\lambda_{n}^{0}$ increase by $1$.

\begin{lemma}
\label{spectrum}
{\rm (1)} For $v\in L_{\mathbb{C}}^2(0,1)$, if $0\in\sigma_{0}$, the spectrum of the operator ${L}_{\alpha}$ is
\begin{equation*}
\sigma\left({L}_{\alpha}\right)=\sigma_{0}\cup\sigma_{2}.
\end{equation*}

{\rm \large{\textcircled{\small{1}}}} If $0\in\sigma_{0}$, $0\notin\sigma_{2}$, the spectrum of the operator ${L}_{\alpha}$ is
\begin{equation*}
\sigma\left({L}_{\alpha}\right)=\{\sigma_{0}\setminus\{(\sigma_{0}\cap\sigma_{2})\cup\{0\}\}\}\cup\{\sigma_{2}\setminus(\sigma_{0}\cap\sigma_{2})\}\cup
\{\sigma_{0}\cap\sigma_{2}\}\cup\{0\},
\end{equation*}
besides,

{\rm (a)} the points $\lambda_{n}^{0}$ from $\sigma_{0}\setminus\{(\sigma_{0}\cap\sigma_{2})\cup\{0\}\}$ have multiplicity $1$ and $\sigma_{0}$ is given by $(\ref{01})$;

{\rm (b)} the points $\mu_{n}$ from $\sigma_{2}\setminus(\sigma_{0}\cap\sigma_{2})$ have multiplicity $1$ and $\sigma_{2}$ is given by $(\ref{02})$;

{\rm (c)} the points $\lambda_{n}^{0}$ from $\sigma_{0}\cap\sigma_{2}$, which coincide with zeros $\mu_{n}$ of the function $Q(\lambda)$ $(\ref{qz})$, have multiplicity $2$;

{\rm (d)} the point $0$  has multiplicity $2$.

{\rm \large{\textcircled{\small{2}}}} If $0\in\sigma_{0}$, $0\in\sigma_{2}$, the spectrum of the operator ${L}_{\alpha}$ is
\begin{equation*}
\sigma\left({L}_{\alpha}\right)=\{\sigma_{0}\setminus(\sigma_{0}\cap\sigma_{2})\}\cup\{\sigma_{2}\setminus(\sigma_{0}\cap\sigma_{2})\}\cup
\{(\sigma_{0}\cap\sigma_{2})\setminus\{0\}\}\cup\{0\},
\end{equation*}
besides,

{\rm (a)} the points $\lambda_{n}^{0}$ from $\sigma_{0}\setminus(\sigma_{0}\cap\sigma_{2})$ have multiplicity $1$ and $\sigma_{0}$ is given by $(\ref{01})$;

{\rm (b)} the points $\mu_{n}$ from $\sigma_{2}\setminus(\sigma_{0}\cap\sigma_{2})$ have multiplicity $1$ and $\sigma_{2}$ is given by $(\ref{02})$;

{\rm (c)} the points $\lambda_{n}^{0}$ from $(\sigma_{0}\cap\sigma_{2})\setminus\{0\}$, which coincide with zeros $\mu_{n}$ of the function $Q(\lambda)$ $(\ref{qz})$, have multiplicity $2$;

{\rm (d)} the point $0$  has multiplicity $3$.

{\rm (2)} For $v\in L_{\mathbb{C}}^2(0,1)$, if $0\in\sigma_{1}$, the spectrum of the operator ${L}_{\alpha}$ is
\begin{equation*}
\sigma\left({L}_{\alpha}\right)=\sigma_{0}\cup\sigma_{2}\cup\{0\}.
\end{equation*}

If $0\in \sigma_{1}$, $0\notin\sigma_{2}$, the spectrum of the operator ${L}_{\alpha}$ is
\begin{equation*}
\sigma\left({L}_{\alpha}\right)=\{\sigma_{0}\setminus(\sigma_{0}\cap\sigma_{2})\}\cup\{\sigma_{2}\setminus(\sigma_{0}\cap\sigma_{2})\}\cup
\{\sigma_{0}\cap\sigma_{2}\}\cup\{0\},
\end{equation*}
besides,

{\rm (a)} the points $\lambda_{n}^{0}$ from $\sigma_{0}\setminus(\sigma_{0}\cap\sigma_{2})$ have multiplicity $1$ and $\sigma_{0}$ is given by $(\ref{01})$;

{\rm (b)} the points $\mu_{n}$ from $\sigma_{2}\setminus(\sigma_{0}\cap\sigma_{2})$ have multiplicity $1$ and $\sigma_{2}$ is given by $(\ref{02})$;

{\rm (c)} the points $\lambda_{n}^{0}$ from $\sigma_{0}\cap\sigma_{2}$, which coincide with zeros $\mu_{n}$ of the function $Q(\lambda)$ $(\ref{qz})$, have multiplicity $2$;

{\rm (d)} the point $0$  has multiplicity $1$.
\end{lemma}

\begin{remark}
\label{finite}
The operator ${L}_{\alpha}$ only has a finite number of non-simple eigenvalues.
\end{remark}
\begin{proof}
Its derivation is similar to that of {\rm\cite[Theorem 4.5] {LY2024}}.
\end{proof}

\subsection{The eigenfunctions of the operator ${L}_{\alpha}$}\

In this subsection, we show the expressions for the eigenfunctions of the operator ${L}_{\alpha}$.
\begin{proposition}
{\rm(1)} Eigenfunctions ${u}_{n}^{{i}}(\alpha,x)$ of the operator ${L}_{\alpha}$ with respect to $\lambda_{n}\in \sigma_{0}$ are given by
\begin{eqnarray*}
{u}_{n}^{{i}}(\alpha,x)={u}_{n}^{{i}},\quad {i}=1,2.
\end{eqnarray*}
{\rm(2)} To each $\mu_{p}\in \sigma_{2}$, the corresponding eigenfunction is
\begin{equation*}
\widetilde{u}_{p}(\alpha,x)=\frac{1}{\sqrt {G^{\prime}(\mu_{p})}}\sum_{\lambda_{n}^{1}\in \sigma_{1}}\frac{v_{n}}{\lambda_{n}^{1}-\mu_{p}}.
\end{equation*}
\end{proposition}
\begin{proof}
${\rm(i)}$ Owing to $\lambda_{n}\in \sigma_{0}$ and $v_{n}^{{i}}=\left<v,{u}_{n}^{{i}}\right>=0$, it means that ${u}_{n}^{{i}}$ satisfies the equation
\begin{equation*}
{L}_{\alpha}{u}_{n}^{{i}}={L}_{0}{u}_{n}^{{i}}+\alpha\left<{u}_{n}^{{i}},v\right>v=\lambda_{n}{u}_{n}^{{i}}
\end{equation*}
and the boundary conditions (\ref{fgh}). So, we have ${u}_{n}^{{i}}(\alpha,x)={u}_{n}^{{i}}$.

${\rm(ii)}$ The residue of the resolvent $R_{{L}_{\alpha}}(\lambda)f$ at $\mu_{p}\in \sigma_{2}$ is
\begin{eqnarray*}
\text{Res}_{\mu_{p}}R_{{L}_{\alpha}}(\lambda)f\!\!&=&\!\!\displaystyle\lim_{\lambda\rightarrow \mu_{p}}(\mu_{p}-\lambda)R_{{L}_{\alpha}}(\lambda)f\\
\!\!&=&\!\!\frac{\alpha}{Q^{\prime}(\mu_{p})}\sum_{\lambda_{n}^{1}\in \sigma_{1}}\frac{v_{n}}{\lambda_{n}^{1}-\mu_{p}}\left(\sum_{\lambda_{n}^{1}\in \sigma_{1}}\frac{f_{n}\overline{v}_{n}}{\lambda_{n}^{1}-\mu_{p}}\right).
\end{eqnarray*}
Due to $Q(\lambda)$ (\ref{qz}) and $G(\lambda)=\sum\limits_{n\in\mathbb{Z}}\frac{\vert v_{n}\vert^{2}}{\lambda_{n}-\lambda}$ , we see
\begin{equation*}
Q^{\prime}(\mu_{p})=\alpha G^{\prime}(\mu_{p}), \quad G^{\prime}(\mu_{p})>0.
\end{equation*}
Set
\begin{equation*}
\widetilde{u}_{p}=\frac{1}{\sqrt {G^{\prime}(\mu_{p})}}\sum_{\lambda_{n}^{1}\in \sigma_{1}}\frac{v_{n}}{\lambda_{n}^{1}-\mu_{p}},
\end{equation*}
we find $\Vert \widetilde{u}_{p}\Vert_{L^{2}}=1$, and  $\widetilde{u}_{p}$ satisfies
\begin{eqnarray*}
{L}_{\alpha}\widetilde{u}_{p}&=&{L}_{0}\widetilde{u}_{p}+\alpha\left<\widetilde{u}_{p},v\right>v\\
&=&\frac{1}{\sqrt {G^{\prime}(\mu_{p})}}\sum_{\lambda_{n}^{1}\in \sigma_{1}}\frac{v_{n}}{\lambda_{n}^{1}-\mu_{p}}\lambda_{n}^{1}+\frac{\alpha}{\sqrt {G^{\prime}(\mu_{p})}}\sum_{\lambda_{n}^{1}\in \sigma_{1}}\frac{\vert v_{n}\vert^{2}}{\lambda_{n}^{1}-\mu_{p}}v\\
&=&\frac{1}{\sqrt {G^{\prime}(\mu_{p})}}\sum_{\lambda_{n}^{1}\in \sigma_{1}}\frac{v_{n}}{\lambda_{n}^{1}-\mu_{p}}(\lambda_{n}^{1}-\mu_{p}+\mu_{p})+\frac{\alpha}{\sqrt {G^{\prime}(\mu_{p})}}\sum_{\lambda_{n}^{1}\in \sigma_{1}}\frac{\vert v_{n}\vert^{2}}{\lambda_{n}^{1}-\mu_{p}}v\\
&=&\frac{1}{\sqrt {G^{\prime}(\mu_{p})}}\sum_{\lambda_{n}^{1}\in \sigma_{1}}\left(v_{n}+\frac{v_{n}\mu_{p}}{\lambda_{n}^{1}-\mu_{p}}\right)+\frac{\alpha}{\sqrt {G^{\prime}(\mu_{p})}}\sum_{\lambda_{n}^{1}\in \sigma_{1}}\frac{\vert v_{n}\vert^{2}}{\lambda_{n}^{1}-\mu_{p}}v\\
&=&\frac{1}{\sqrt {G^{\prime}(\mu_{p})}}\left[\sum_{\lambda_{n}\in \sigma(L_{0})}v_{n}+{\alpha}\sum_{\lambda_{n}^{1}\in \sigma_{1}}\frac{\vert v_{n}\vert^{2}}{\lambda_{n}^{1}-\mu_{p}}v\right]+\frac{\mu_{p}}{\sqrt {G^{\prime}(\mu_{p})}}\sum_{\lambda_{n}^{1}\in \sigma_{1}}\frac{v_{n}}{\lambda_{n}^{1}-\mu_{p}}\\
&=&\frac{1}{\sqrt {G^{\prime}(\mu_{p})}}\left[1+{\alpha}\sum_{\lambda_{n}^{1}\in \sigma_{1}}\frac{\vert v_{n}\vert^{2}}{\lambda_{n}^{1}-\mu_{p}}\right]v+{\mu_{p}}\widetilde{u}_{p}\\
&=&{\mu_{p}}\widetilde{u}_{p}
\end{eqnarray*}
and the boundary conditions (\ref{fgh}). Then, for $\mu_{p}\in \sigma_{2}$,  $\widetilde{u}_{p}(\alpha,x)=\widetilde{u}_{p}$ is an eigenfunction of ${L}_{\alpha}$.
\end{proof}

\begin{remark}
For $\lambda_{n}\in\sigma_{0}\cap\sigma_{2}$, the eigenfunctions ${u}_{n}^{{i}}(\alpha,x)$ and $\widetilde{u}_{n}(\alpha,x)$ are linearly independent.
\end{remark}

\section{The inverse problem}

In this section, we study the inverse problems for the operator ${L}_{\alpha}$. We first give the Ambarzumyan theorem as follows.
\begin{theorem}
\label{A}
For $v\in L_{\mathbb{C}}^2(0,1)$, if $\sigma\left({L}_{0}\right)=\sigma\left({L}_{\alpha}\right)$, then $v=0$.
\end{theorem}
\begin{proof}
Since the function $v\in L_{\mathbb{C}}^2(0,1)$ can be expressed by its Fourier series, i.e.,
\begin{equation}
v(x)=\sum\limits_{n\in \mathbb{Z}}v_{n}.\label{vu}
\end{equation}
From Lemma \ref{spectrum}, $\sigma\left({L}_{0}\right)=\sigma_{0}\cup\sigma_{1}$ and $\sigma\left({L}_{0}\right)=\sigma\left({L}_{\alpha}\right)$, we get $\sigma\left({L}_{0}\right)=\sigma_{0}=\sigma\left({L}_{\alpha}\right)$ ($0\in \sigma_{0}$ which has multiplicity 2). Then, on account of the definition of $\sigma_{0}$ in $(\ref{01})$, one has ${v_{n}}=0$ for $n\in \mathbb{Z}$. Therefore, the Fourier series  (\ref{vu}) implies $v=0$.
\end{proof}

Next, we investigate the spectral data for reconstructing the non-local potential $\{\alpha, v\}$.
\begin{lemma}
The function $Q(\lambda)$ $(\ref{qz})$ satisfies the following identity
\begin{equation}
Q(\lambda)=\frac{\Delta(\alpha, \lambda)}{\Delta(0, \lambda)}.\label{ert}
\end{equation}
\end{lemma}
\begin{proof}
From (\ref{qz}) and Lemma \ref{RLA}, we obtain
\begin{equation*}
Q(\lambda)=1+\alpha\left<R_{{L}_{0}}(\lambda)v,v\right>,
\end{equation*}
then using expression of $R_{{L}_{0}}(\lambda)$ in Lemma \ref{rel0}, $\Delta(\alpha, \lambda)$ in (\ref{deltaa}) and ${F}(\lambda)$ in (\ref{lanm}), we calculate
\begin{eqnarray*}
\left<R_{{L}_{0}}(\lambda)v,v\right>
\!\!&=&\!\!\frac{\mathrm{i}}{\Delta(0,\lambda)}\left\{\bigg<\int_{0}^{x}\mathrm{i}\lambda[y_{3}(x,\lambda)y_{1}(-t,\lambda)y_{2}(-1,\lambda)+y_{2}(1,\lambda)\right.\\
\!\!&&\!\!\left.y_{3}(x-t,\lambda)\right]v(t)\mathrm{d}t+\int_{0}^{1}\left[-y_{1}(x,\lambda)y_{1}(-t,\lambda)y_{1}(-1,\lambda)\right.\\
\!\!&&\!\!\left.+\mathrm{i}\lambda y_{2}(x,\lambda)y_{1}(-t,\lambda)y_{3}(-1,\lambda)+y_{1}(1-t,\lambda)y_{1}(x,\lambda)\right]v(t)\mathrm{d}t\\
\!\!&&\!\!\left.+\int_{x}^{1}(-\mathrm{i}\lambda) \left[y_{1}(x,\lambda)y_{3}(-t,\lambda)y_{2}(-1,\lambda)+y_{2}(x,\lambda)y_{2}(-t,\lambda)\right.\right.\\
\!\!&&\!\!\left.\left.y_{2}(-1,\lambda)\right]v(t)\mathrm{d}t,v\bigg>\right\}\\
\!\!&=&\!\!\frac{\mathrm{i}}{\Delta(0,\lambda)}\left[\mathrm{i}\lambda y_{2}(1,\lambda)m(\lambda)+\widetilde{\nu}_{1}(-\lambda)\widetilde{y}_{1}^{*}(\lambda)-\widetilde{y}_{1}(\lambda)\widetilde{\nu}_{1}^{*}(-\lambda)\right.\\
\!\!&&\!\!\left.-\mathrm{i}\lambda y_{2}(-1,\lambda)m^{*}(\lambda)]\right.\\
\!\!&=&\!\!\frac{\mathrm{i}}{\Delta(0,\lambda)}[{F}(\lambda)-{F}^{*}(\lambda)]\\
\!\!&=&\!\!\frac{\Delta(\alpha, \lambda)-\Delta(0, \lambda)}{\alpha\Delta(0, \lambda)}.
\end{eqnarray*}
Thus,
\begin{equation*}
Q(\lambda)=1+\alpha\frac{\Delta(\alpha, \lambda)-\Delta(0, \lambda)}{\alpha\Delta(0, \lambda)}=\frac{\Delta(\alpha, \lambda)}{\Delta(0, \lambda)}.
\end{equation*}
\end{proof}

\begin{lemma}
The multiplicative expansions of characteristic functions $\Delta(0, \lambda)$ and $\Delta(\alpha, \lambda)$ are
\begin{equation}
\Delta(0, \lambda)=2\sum_{n=1}^{\infty}\frac{(-1) ^{n}\lambda^{2n}}{(6n-2)!}=-\frac{2}{4!}\lambda^{2}\prod _{n\in\mathbb{N}^{*}, \lambda_{n}(0)\neq0}\left(1-\frac{\lambda^{2}}{\lambda_{n}^{2}(0)}\right), \label{yuh}
\end{equation}
\begin{equation}
\Delta(\alpha, \lambda)=\left\{
\begin{aligned}
b_{2}\lambda^{2}\prod _{n\in \mathbb{Z}, \lambda_{n}(\alpha)\neq0}\left(1-\frac{\lambda}{\lambda_{n}(\alpha)}\right), & & {for\quad0\in\sigma_{0}, \quad0\notin\sigma_{2}},\\
b_{3}\lambda^{3}\prod _{n\in \mathbb{Z}, \lambda_{n}(\alpha)\neq0}\left(1-\frac{\lambda}{\lambda_{n}(\alpha)}\right), & & {for\quad0\in\sigma_{0}, \quad0\in\sigma_{2}},\\
b_{1}\lambda \prod _{n\in \mathbb{Z}, \lambda_{n}(\alpha)\neq0}\left(1-\frac{\lambda}{\lambda_{n}(\alpha)}\right), & & {for\quad0\in\sigma_{1}, \quad0\notin\sigma_{2}},\\
\end{aligned}
\right.\label{sdf}
\end{equation}
respectively, where
\begin{eqnarray}
b_{1}\!\!&=&\!\!\alpha\left[\int_{0}^{1}\int_{0}^{1}\left(\frac{t^3}{3!}+\frac{(1-x)^3}{3!}+\frac{x^3}{3!}+\frac{(1-t)^3}{3!}\right){v}(t)\overline{v}(x)\mathrm{d}t\mathrm{d}x\right.\notag\\
\!\!&&\!\!\left.-\int_{0}^{1}\int_{0}^{x}(x-t)^2{v}(t)\overline{v}(x)\mathrm{d}t\mathrm{d}x\right],\label{b1}
\end{eqnarray}
\begin{eqnarray}
b_{2}\!\!&=&\!\!-\frac{2}{4!}+\mathrm{i}\alpha\left[\int_{0}^{1}\int_{0}^{1}\left(\frac{t^6+(1-x)^6}{6!}-\frac{x^6+(1-t)^6}{6!}+\frac{t^3}{3!}\frac{(1-x)^3}{3!}\right.\right.\notag\\
\!\!&&\!\!\left.\left.-\frac{x^3}{3!}\frac{(1-t)^3}{3!}\right){v}(t)\overline{v}(x)\mathrm{d}t\mathrm{d}x+\int_{0}^{1}\int_{0}^{x}\frac{2}{5!}(x-t)^5{v}(t)\overline{v}(x)\mathrm{d}t\mathrm{d}x\right],\label{er}
\end{eqnarray}
\begin{eqnarray}
b_{3}\!\!&=&\!\!\alpha\left[\int_{0}^{1}\int_{0}^{x}\bigg(\frac{2(x-t)^8}{8!}+\frac{(x-t)^2}{7!}\bigg){v}(t)\overline{v}(x)\mathrm{d}t\mathrm{d}x\right.\notag\\
\!\!&&\!\!\left.-\int_{0}^{1}\int_{0}^{1}\left(\frac{x^9+(1-t)^9}{9!}+\frac{t^9+(1-x)^9}{9!}+\frac{x^6}{6!}\frac{(1-t)^3}{3!}\right.\right.\notag\\
\!\!&&\!\!\left.\left.+\frac{x^3}{3!}\frac{(1-t)^6}{6!}+\frac{t^6}{6!}\frac{(1-x)^3}{3!}+\frac{t^3}{3!}\frac{(1-x)^6}{6!}\right){v}(t)\overline{v}(x)\mathrm{d}t\mathrm{d}x\right],\label{ty}
\end{eqnarray}
and $\lambda_{n}(\alpha)$ represents the zero of $\Delta(\alpha, \lambda)$.
\end{lemma}
\begin{proof}
According to (\ref{Delta0}) and the Hadamard theorem on factorization {\cite{BYL1997}}, we can obtain the multiplicative expansions of characteristic functions $\Delta (0,\lambda)$ (\ref{yuh}).

For the convenience of calculation, we introduce  $k=\sqrt[3]{\lambda}$.
From (\ref{0k}), ${F}(\lambda)$ in (\ref{lanm}) and the expressions of $y_{1}(x,\lambda)$, $y_{2}(x,\lambda)$ and $y_{3}(x,\lambda)$ in (\ref{yul}), we have
\begin{equation*}
\Delta(0, \lambda)=2\sum_{n=1}^{\infty}\frac{(\mathrm{i}k)^{6n}}{(6n-2)!},
\end{equation*}
\begin{equation*}
F(\lambda)=a_{0}(v)+a_{1}(v)k^{3}+a_{2}(v)k^{6}+a_{3}(v)k^{9}+\cdots,
\end{equation*}
where
\begin{eqnarray*}
&&a_{0}(v)=\bigg<1,v(x)\bigg>\bigg<v(1-t),1\bigg>,\\
&&a_{1}(v)=\bigg<1,v(x)\bigg>\bigg<v(1-t),\frac{(-\mathrm{i}t)^3}{3!}\bigg>+\bigg<\frac{(\mathrm{i}x)^3}{3!},v(x)\bigg>\bigg<v(1-t),1\bigg>\\
&&\qquad\qquad-\mathrm{i}\left<\int_{0}^{x}\frac{\mathrm{i}^2(x-t)^2}{2!}v(t)\mathrm{d}t,v(x)\right>,\\
&&a_{2}(v)=\bigg<1,v(x)\bigg>\bigg<v(1-t),\frac{(-\mathrm{i}t)^6}{6!}\bigg>+\bigg<\frac{(\mathrm{i}x)^3}{3!},v(x)\bigg>\bigg<v(1-t),\frac{(-\mathrm{i}t)^3}{3!}\bigg>\\
&&\qquad\qquad+\bigg<\frac{(\mathrm{i}x)^6}{6!},v(x)\bigg>\bigg<v(1-t),1\bigg>-\mathrm{i}\left<\int_{0}^{x}\frac{\mathrm{i}^5(x-t)^5}{5!}v(t)\mathrm{d}t,v(x)\right>\\
&&\qquad\qquad-\frac{\mathrm{i}^4}{4!}\left<\int_{0}^{x}\frac{\mathrm{i}^2(x-t)^2}{2!}v(t)\mathrm{d}t,v(x)\right>,\\
&&a_{3}(v)=\bigg<1,v(x)\bigg>\bigg<v(1-t),\frac{(-\mathrm{i}t)^9}{9!}\bigg>+\bigg<\frac{(\mathrm{i}x)^3}{3!},v(x)\bigg>\bigg<v(1-t),\frac{(-\mathrm{i}t)^6}{6!}\bigg>\\
&&\qquad\qquad+\bigg<\frac{(\mathrm{i}x)^6}{6!},v(x)\bigg>\bigg<v(1-t),\frac{(-\mathrm{i}t)^3}{3!}\bigg>+\bigg<\frac{(\mathrm{i}x)^9}{9!},v(x)\bigg>\bigg<v(1-t),1\bigg>\\
&&\qquad\qquad-\mathrm{i}\left<\int_{0}^{x}\frac{\mathrm{i}^8(x-t)^8}{8!}v(t)\mathrm{d}t,v(x)\right>-\frac{\mathrm{i}^4}{4!}\left<\int_{0}^{x}\frac{\mathrm{i}^5(x-t)^5}{5!}v(t)\mathrm{d}t,v(x)\right>\\
&&\qquad\qquad-\frac{\mathrm{i}^7}{7!}\left<\int_{0}^{x}\frac{\mathrm{i}^2(x-t)^2}{2!}v(t)\mathrm{d}t,v(x)\right>,
\end{eqnarray*}
and $a_{n}(v)$ can be calculated similarly. Therefore,  the equality (\ref{deltaa}) shows
\begin{eqnarray*}
\Delta(\alpha,\lambda)
\!\!&=&\!\!\mathrm{i}\alpha[a_{0}(v)-\overline{a}_{0}(v)]+\mathrm{i}\alpha[a_{1}(v)-\overline{a}_{1}(v)]k^{3}+\bigg[\mathrm{i}\alpha(a_{2}(v)-\overline{a}_{2}(v))+\frac{2}{4!}\mathrm{i}^6\bigg]k^{6}\\
\!\!&&\!\!+\cdots+\mathrm{i}\alpha[a_{q}(v)-\overline{a}_{q}(v)]k^{3q}+\bigg[\mathrm{i}\alpha(a_{q+1}(v)-\overline{a}_{q+1}(v))+\frac{2\mathrm{i}^{3q+3}}{(3q+1)!}\bigg]k^{3q+3}\\
\!\!&&\!\!+\cdots\\
\!\!&=&\!\!b_{0}+b_{1}k^{3}+b_{2}k^{6}+b_{3}k^{9}+\cdots+b_{n}k^{3n}+\cdots,
\end{eqnarray*}
where, $k=\sqrt[3]{\lambda}$, $b_{1}$, $b_{2}$ and $b_{3}$ are defined in (\ref{b1}), (\ref{er}) and (\ref{ty}), respectively.

From the Hadamard theorem on factorization {\cite{BYL1997}}, we suppose
\begin{equation}
\Delta(\alpha, \lambda)=bk^{q}\mathrm{e}^{ck}\prod _{n\in \mathbb{Z}, k_{n}(\alpha)\neq0}\left(1-\frac{k^3}{k_{n}^3(\alpha)}\right),\label{qw}
\end{equation}
where $k_{n}(\alpha)=\sqrt[3]{\lambda_{n}(\alpha)}$.

For $0\in\sigma_{0}$, $0\notin\sigma_{2}$, that is, $0$ is the eigenvalue of the operator ${L}_{\alpha}$ which has multiplicity $2$, then $b_{0}=b_{1}=0$. The equality (\ref{qw}) shows $q=6$.  Then from $\Delta^{\prime}_{k}(\alpha, 0)=0$, we find $c=0$ and
\begin{equation*}
\Delta(\alpha, \lambda)=b_{2}k^{6}\prod _{n\in \mathbb{Z}, k_{n}(\alpha)\neq0}\left(1-\frac{k^3}{k_{n}^3(\alpha)}\right)=b_{2}\lambda^{2}\prod _{n\in \mathbb{Z}, \lambda_{n}(\alpha)\neq0}\left(1-\frac{\lambda}{\lambda_{n}(\alpha)}\right).
\end{equation*}

For $0\in\sigma_{0}$, $0\in\sigma_{2}$, that is, $0$ is the eigenvalue of the operator ${L}_{\alpha}$ which has multiplicity $3$, then $b_{0}=b_{1}=b_{2}=0$.  According to the equality (\ref{qw}), we discover $q=9$.  Then the identity $\Delta^{\prime}_{k}(\alpha, 0)=0$ indicate $c=0$ and
\begin{equation*}
\Delta(\alpha, \lambda)=b_{3}k^{9}\prod _{n\in \mathbb{Z}, k_{n}(\alpha)\neq0}\left(1-\frac{k^3}{k_{n}^3(\alpha)}\right)=b_{3}\lambda^{3}\prod _{n\in \mathbb{Z}, \lambda_{n}(\alpha)\neq0}\left(1-\frac{\lambda}{\lambda_{n}(\alpha)}\right).
\end{equation*}

For $0\in\sigma_{1}$, $0\notin\sigma_{2}$, that is, $0$ is the  simple eigenvalue of the operator ${L}_{\alpha}$, then $b_{0}=0$. The equality (\ref{qw}) shows $q=3$.  Then from $\Delta^{\prime}_{k}(\alpha, 0)=0$, one gets $c=0$ and
\begin{equation*}
\Delta(\alpha, \lambda)=b_{1}k^{3}\prod _{n\in \mathbb{Z}, k_{n}(\alpha)\neq0}\left(1-\frac{k^3}{k_{n}^3(\alpha)}\right)=b_{1}\lambda\prod _{n\in \mathbb{Z}, \lambda_{n}(\alpha)\neq0}\left(1-\frac{\lambda}{\lambda_{n}(\alpha)}\right).
\end{equation*}
\end{proof}

\begin{lemma}
\label{sd}
For $v\in L_{\mathbb{C}}^2(0,1)$, $\Vert v\Vert=1$, the spectra data $\{\alpha, \vert v_{p}\vert^{2}\}$ can be recovered by the spectra $\sigma\left({L}_{0}\right)$ and $\sigma\left({L}_{\alpha}\right)$.
\end{lemma}
\begin{proof}
${(\rm i)}$ For $0\in\sigma_{0}$, $0\notin\sigma_{2}$, combining (\ref{qz}), (\ref{ert}), (\ref{yuh}) and (\ref{sdf}), we find
\begin{eqnarray}
Q(\lambda)\!\!&=&\!\!\frac{\Delta(\alpha,\lambda)}{\Delta(0,\lambda)}\notag\\
\!\!&=&\!\!1+\alpha\sum_{\lambda_{n}^{1}\in\sigma_{1}}\frac{\vert v_{n}\vert^{2}}{\lambda_{n}^{1}-\lambda}\notag\\
\!\!&=&\!\!\frac{b_{2}\lambda^{2}\prod\limits_{n\in \mathbb{Z}, \lambda_{n}(\alpha)\neq0}\left(1-\frac{\lambda}{\lambda_{n}(\alpha)}\right)}{-\frac{2}{4!}\lambda^{2}\prod\limits_{n\in \mathbb{Z}, \lambda_{n}(0)\neq0}\left(1-\frac{\lambda}{\lambda_{n}(0)}\right)}\notag\\
\!\!&=&\!\!-\frac{4!b_{2}\prod\limits_{\mu_n\in\sigma_{2}}\left(1-\frac{\lambda}{\mu_n}\right)}{2\prod\limits_{\lambda_{n}^{1}\in\sigma_{1}}\left(1-\frac{\lambda}{\lambda_n^1}\right)}.\label{opi}
\end{eqnarray}
Substituting $\lambda=\mathrm{i}y$ and passing to the limit as $y\rightarrow \infty$, one gets
\begin{equation*}
\frac{1}{b_{2}}=\lim_{y\rightarrow \infty}-\frac{4!\prod\limits_{\mu_n\in\sigma_{2}}\left(1-\frac{\mathrm{i}y}{\mu_n}\right)}{2\prod\limits_{\lambda_{n}^{1}\in\sigma_{1}}\left(1-\frac{\mathrm{i}y}{\lambda_n^1}\right)}.
\end{equation*}
Calculating the residue at the point $\lambda_{n}^{1}$ in equality (\ref{opi}), we have
\begin{eqnarray*}
&&\alpha \vert v_{n}\vert^{2}=-\frac{4!b_{2}}{2}\frac{\lambda_n^1}{\mu_{n}}(\mu_{n}-\lambda_n^1)\prod\limits_{q\neq n}\frac{\lambda_{q}^{1}}{\mu_{q}}\left(1-\frac{\lambda_{q}^{1}-\mu_{q}}{\lambda_{q}^{1}-\lambda_{n}^{1}}\right).
\end{eqnarray*}
So, $\Vert v\Vert=1$ indicates the spectra data $\{\alpha, \vert v_{p}\vert^{2}\}$.

${(\rm ii)}$ For $0\in\sigma_{0}$, $0\in\sigma_{2}$, combining (\ref{qz}), (\ref{ert}), (\ref{yuh}) and (\ref{sdf}), we get
\begin{eqnarray}
Q(\lambda)\!\!&=&\!\!\frac{\Delta(\alpha,\lambda)}{\Delta(0,\lambda)}\notag\\
\!\!&=&\!\!1+\alpha\sum_{\lambda_{n}^{1}\in\sigma_{1}}\frac{\vert v_{n}\vert^{2}}{\lambda_{n}^{1}-\lambda}\notag\\
\!\!&=&\!\!\frac{b_{3}\lambda^{3}\prod\limits_{n\in \mathbb{Z}, \lambda_{n}(\alpha)\neq0}\left(1-\frac{\lambda}{\lambda_{n}(\alpha)}\right)}{-\frac{2}{4!}\lambda^{2}\prod\limits_{n\in \mathbb{Z}, \lambda_{n}(0)\neq0}\left(1-\frac{\lambda}{\lambda_{n}(0)}\right)}\notag\\
\!\!&=&\!\!-\frac{4!b_{3}\lambda\prod\limits_{\mu_n\in{\sigma_{2}\setminus\{0\}}}\left(1-\frac{\lambda}{\mu_n}\right)}{2\prod\limits_{\lambda_{n}^{1}\in\sigma_{1}}\left(1-\frac{\lambda}{\lambda_n^1}\right)}.\label{qzz}
\end{eqnarray}
Substituting $\lambda=\mathrm{i}y$ and passing to the limit as $y\rightarrow \infty$, one has
\begin{equation*}
\frac{1}{b_{3}}=\lim_{y\rightarrow \infty}-\frac{4!\mathrm{i}y\prod\limits_{\mu_n\in\sigma_{2}\setminus\{0\}}\left(1-\frac{\mathrm{i}y}{\mu_n}\right)}{2\prod\limits_{\lambda_{n}^{1}\in\sigma_{1}}\left(1-\frac{\mathrm{i}y}{\lambda_n^1}\right)}.
\end{equation*}
Computing the residue at the point $\lambda_{n}^{1}$ in equality (\ref{qzz}), one obtains
\begin{eqnarray*}
&&\alpha \vert v_{n}\vert^{2}=-\frac{4!b_{3}}{2}\frac{{\lambda_{n}^1}^2}{\mu_{n}}(\mu_{n}-\lambda_n^1)\prod\limits_{q\neq n}\frac{\lambda_{q}^{1}}{\mu_{q}}\left(1-\frac{\lambda_{q}^{1}-\mu_{q}}{\lambda_{q}^{1}-\lambda_{n}^{1}}\right).
\end{eqnarray*}
Due to $\Vert v\Vert=1$, the spectra data $\{\alpha, \vert v_{p}\vert^{2}\}$ can be reconstructed.\\

${(\rm iii)}$ For $0\in\sigma_{1}$, $0\notin\sigma_{2}$, combining (\ref{qz}), (\ref{ert}), (\ref{yuh}) and (\ref{sdf}), we have
\begin{eqnarray}
Q(\lambda)\!\!&=&\!\!\frac{\Delta(\alpha,\lambda)}{\Delta(0,\lambda)}\notag\\
\!\!&=&\!\!1+\alpha\sum_{\lambda_{n}^{1}\in\sigma_{1}}\frac{\vert v_{n}\vert^{2}}{\lambda_{n}^{1}-\lambda}\notag\\
\!\!&=&\!\!\frac{b_{1}\lambda\prod\limits_{n\in \mathbb{Z}, \lambda_{n}(\alpha)\neq0}\left(1-\frac{\lambda}{\lambda_{n}(\alpha)}\right)}{-\frac{2}{4!}\lambda^{2}\prod\limits_{n\in \mathbb{Z}, \lambda_{n}(0)\neq0}\left(1-\frac{\lambda}{\lambda_{n}(0)}\right)}\notag\\
\!\!&=&\!\!-\frac{4!b_{1}\prod\limits_{\mu_n\in\sigma_{2}}\left(1-\frac{\lambda}{\mu_n}\right)}{2\lambda\prod\limits_{\lambda_{n}^{1}\in\sigma_{1}\setminus\{0\}}\left(1-\frac{\lambda}{\lambda_n^1}\right)}.\label{1not2}
\end{eqnarray}
Substituting $\lambda=\mathrm{i}y$ and passing to the limit as $y\rightarrow \infty$, one gets
\begin{equation*}
\frac{1}{b_{1}}=\lim_{y\rightarrow \infty}-\frac{4!\prod\limits_{\mu_n\in\sigma_{2}}\left(1-\frac{\mathrm{i}y}{\mu_n}\right)}{2\mathrm{i}y\prod\limits_{\lambda_{n}^{1}\in\sigma_{1}\setminus\{0\}}\left(1-\frac{\mathrm{i}y}{\lambda_n^1}\right)}.
\end{equation*}
Calculating the residue at the point $\lambda_{n}^{1}$ in equality (\ref{1not2}), we find
\begin{eqnarray*}
&&\alpha \vert v_{n}\vert^{2}=-\frac{4!b_{1}}{2}\frac{\mu_{n}-\lambda_n^1}{\mu_{n}}\prod\limits_{q\neq n}\frac{\lambda_{q}^{1}}{\mu_{q}}\left(1-\frac{\lambda_{q}^{1}-\mu_{q}}{\lambda_{q}^{1}-\lambda_{n}^{1}}\right).
\end{eqnarray*}
Thus, $\Vert v\Vert=1$ indicates the spectra data $\{\alpha, \vert v_{p}\vert^{2}\}$.
\end{proof}

\begin{theorem}
\label{four spectra}
For $v\in L_{\mathbb{C}}^2(0,1)$, $\Vert v\Vert=1$, the non-local potential $\{\alpha, v\}$ can be recovered by four spectra $\sigma\left({L}_{0}\right)$, $\sigma\left({L}_{\alpha}(v)\right)$, $\sigma\left({L}_{\alpha}(v+g)\right)$ and $\sigma\left({L}_{\alpha}(v+\mathrm{i}g)\right)$, where $g(x)=1-x$.
\end{theorem}
\begin{proof}
From Lemma \ref{sd}, we see that the numbers $\{\alpha, \vert v_{p}\vert^2\}$ can be recovered by $\sigma\left({L}_{0}\right)$ and $\sigma\left({L}_{\alpha}(v)\right)$. In the same way, we can calculate $\alpha\vert v_{n}+g_{n}\vert^2$ from $\sigma\left({L}_{0}\right)$ and $\sigma\left({L}_{\alpha}(v+g)\right)$. Let $\Re\beta$ denote the real part of $\beta$ and $\Im\beta$ denote the imaginary part of $\beta$. Due to
\begin{equation}
\alpha\vert v_{n}+g_{n}\vert^2=\alpha\left(\vert v_{n}\vert^2+2\Re (v_{n}\overline{g}_{n})+\vert g_{n}\vert^2\right),
\end{equation}
where
\begin{equation*}
g_{n}:=\left\{\begin{array}{ll}
g_{0}^{\mathrm{1}}u_{0}^{\mathrm{1}}+g_{0}^{\mathrm{2}}u_{0}^{\mathrm{2}}, & {for\quad n=0},\\
g_{n}^{\mathrm{1}}u_{n}^{\mathrm{1}},&  {for\quad n\in \mathbb{Z}\setminus\{0\}},
\end{array}
\right.
\end{equation*}
and $g_{n}^{{i}}:=\left<g,u_{n}^{{i}}\right>$, we find that the numbers $\Re (v_{n}\overline{g}_{n})$ are unambiguously recovered by three spectra $\sigma\left({L}_{0}\right)$, $\sigma\left({L}_{\alpha}(v)\right)$ and  $\sigma\left({L}_{\alpha}(v+g)\right)$. Analogously, from $\sigma\left({L}_{0}\right)$, $\sigma\left({L}_{\alpha}(v)\right)$ and  $\sigma\left({L}_{\alpha}(v+\mathrm{i}g)\right)$, we get $\Im (v_{n}\overline{g}_{n})$. So, the numbers $v_{n}\overline{g}_{n}$ and $ v_{n}$ are unambiguously calculated by four spectra $\sigma\left({L}_{0}\right)$, $\sigma\left({L}_{\alpha}(v)\right)$, $\sigma\left({L}_{\alpha}(v+g)\right)$ and $\sigma\left({L}_{\alpha}(v+\mathrm{i}g)\right)$. Therefore, the function $v(x)$ is reconstructed by its Fourier series (\ref{vu}).
\end{proof}

\section{Appendix}

In this section, we analyze the  multiplicity for eigenvalues of the operator $\widetilde{L}_{0}$ defined by 
\begin{equation*}
\left(\widetilde{L}_{0}y\right)(x):=\mathrm{i}y^{\prime\prime\prime}(x),
\end{equation*}
whose domain $\mathcal{D}\left(\widetilde{L}_{0}\right)$ consists of the functions $y\in W_{3}^{2}(0,1)$ satisfying the boundary conditions
\begin{equation*}
\cos\gamma y(0)-\mathrm{i}\sin \gamma y^{\prime\prime}(0)=0, \quad y^{\prime}(1)=\mathrm{e}^{\mathrm{i}\phi}y^{\prime}(0), \quad \cos\beta y(1)-\mathrm{i}\sin \beta y^{\prime\prime}(1)=0, 
\end{equation*}
where $\gamma$, $\beta\in\mathbb{R}$, $\phi\in[0,2\pi]$.
\begin{theorem}
\label{distribution}
The multiplicity for each eigenvalue of the operator $\widetilde{L}_{0}$ is at most $2$.\\
{\rm (1)} For $\cos^2\gamma+\cos^2\beta+(\mathrm{e}^{\mathrm{i}\phi}-1)^2\neq0$, the operator $\widetilde{L}_{0}$ only has simple eigenvalues.\\
{\rm (2)} For $\cos^2\gamma+\cos^2\beta+(\mathrm{e}^{\mathrm{i}\phi}-1)^2=0$,
the boundary conditions are
\begin{equation*}
y^{\prime\prime}(0)=0, \quad y^{\prime}(1)=y^{\prime}(0), \quad y^{\prime\prime}(1)=0,
\end{equation*}
and the operator $\widetilde{L}_{0}$ has simple eigenvalues, except for the eigenvalue $0$ which has multiplicity $2$.
\end{theorem}
\begin{proof}
Since the unique solution of (\ref{L0})-(\ref{initial}) is (\ref{solutionl0}), then from the boundary conditions (\ref{beta}), we get the following system of linear equations,
\begin{equation*}
\left\{
\begin{array}{lr}
0=\cos\gamma c_{1} -\mathrm{i}\sin\gamma c_{3}, & \\
0=c_{1}y_{1}^{\prime}(1,\lambda)+c_{2}\left(y_{2}^{\prime}(1,\lambda)-\mathrm{e}^{\mathrm{i}\phi}\right)+c_{3}y_{3}^{\prime}(1,\lambda), & \\
0=\left[\cos\beta\left(c_{1}y_{1}(1,\lambda)+c_{2}y_{2}(1,\lambda)+c_{3}y_{3}(1,\lambda)\right)\right.\\
\quad \quad \left.-\mathrm{i}\sin\beta (c_{1}y_{1}^{\prime \prime}(1,\lambda)+c_{2}y_{2}^{\prime \prime}(1,\lambda)+c_{3}y_{3}^{\prime \prime}(1,\lambda))\right]. & \\
\end{array}
\right.
\end{equation*}

(I) For $\cos\gamma\neq0$. Due to $c_{1}= \mathrm{i}\tan\gamma c_{3}$, the coefficient matrix for $c_{2}$, $c_{3}$ is
\begin{equation*}
M_{1}(\lambda):=\left(
\begin{array}{cc}
y_{2}^{\prime}(1,\lambda)-\mathrm{e}^{\mathrm{i}\phi} & \mathrm{i}\tan\gamma y_{1}^{\prime}(1,\lambda)+y_{3}^{\prime}(1,\lambda)  \\
\cos\beta y_{2}(1,\lambda)-\mathrm{i}\sin\beta y_{2}^{\prime \prime}(1,\lambda) & \mathrm{i}tan\gamma r_{1}+r_{2} %
\end{array}%
\right),
\end{equation*}
where
\begin{equation*}
r_{1}=\cos\beta y_{1}(1,\lambda)-\mathrm{i}\sin\beta y_{1}^{\prime \prime}(1,\lambda), \\ \quad r_{2}=\cos\beta y_{3}(1,\lambda)-\mathrm{i}\sin\beta y_{3}^{\prime \prime}(1,\lambda).
\end{equation*}
From the relation between the rank of the matrix and the multiplicity of eigenvalues,
we can see that $\lambda$ is the root of $\Delta_{1}(\lambda)=\det M_{1}(\lambda)=0$,
and the operator $\widetilde{L}_{0}$ has eigenvalues of multiplicity $2$ if and only if $M_{1}(\lambda)$ is a zero matrix, that is,
\begin{equation*}
M_{1}(\lambda):=\left(
\begin{array}{cc}
0 & 0  \\
0 & 0 %
\end{array}%
\right).
\end{equation*}
According to (1) in Lemma \ref{entire}, we get
\begin{equation}
\left\{
\begin{array}{lr}
0=y_{1}(1,\lambda)-\mathrm{e}^{\mathrm{i}\phi}, & \\
0=\lambda \tan\gamma  y_{3}(1,\lambda)+y_{2}(1,\lambda), & \\
0=\cos\beta y_{2}(1,\lambda)-\lambda\sin\beta y_{3}(1,\lambda), & \\
0=\mathrm{i}\tan\gamma\left(\cos\beta y_{1}(1,\lambda)-\lambda \sin\beta y_{2}(1,\lambda)\right)
+\cos\beta y_{3}(1,\lambda)-\mathrm{i}\sin\beta y_{1}(1,\lambda).
\end{array}
\right.\label{3.6}
\end{equation}
By a simple  calculation, one obtains
\begin{equation*}
\lambda y_{3}(1,\lambda)(\cos\beta \tan\gamma+\sin\beta)=0.
\end{equation*}

(i) If $\lambda=0$, then we have $y_{2}(1,\lambda)=0$, which contradicts $\displaystyle\lim_{\lambda\rightarrow 0}y_{2}(1,\lambda)=1$.

(ii) If $y_{3}(1,\lambda)=0$, then one obtains $y_{2}(1,\lambda)=0$. Combining $y_{3}(1,\lambda)=0$ and $y_{2}(1,\lambda)=0$, we get $\lambda=0$, which contradicts $\displaystyle\lim_{\lambda\rightarrow 0}y_{2}(1,\lambda)=1$.

(iii) If $\cos\beta \tan\gamma+\sin\beta=0$, it needs to be discussed in two cases.

\large{\textcircled{\small{i}}} If $\cos\beta=0$, it follows that $\sin\beta=0$, this leads to a contradiction.

\large{\textcircled{\small{ii}}} If $\cos\beta\neq0$, one gets $\tan\gamma=-\tan\beta$. According to (\ref{3.6}), it is calculated that
\begin{equation}
\cos\beta y_{3}(1,\lambda)-2\mathrm{i}\sin\beta y_{1}(1,\lambda)+\mathrm{i}\lambda\sin\beta \tan\beta y_{2}(1,\lambda)=0.\label{3.7}
\end{equation}
If $\sin\beta=0$, due to the equation (\ref{3.7}), we get $y_{3}(1,\lambda)=0$. As discussed in the case of (ii) in (I), one has $y_{2}(1,\lambda)=0$; considering $y_{3}(1,\lambda)=0$ and $y_{2}(1,\lambda)=0$, 
we obtain $\lambda=0$, which contradicts $\displaystyle\lim_{\lambda\rightarrow 0}y_{2}(1,\lambda)=1$. If $\sin\beta\neq0$, according to Proposition \ref{2.1} and (8) in Lemma \ref{entire}, we see
\begin{equation*}
\left(\mathrm{e}^{\frac{\sqrt{3}}{2}\sqrt[3]{\lambda}}+\mathrm{e}^{-\frac{\sqrt{3}}{2}\sqrt[3]{\lambda}}\right)\left(\mathrm{e}^{\frac{3}{2}\mathrm{i}\sqrt[3]{\lambda}}+\mathrm{e}^{-\frac{3}{2}\mathrm{i}\sqrt[3]{\lambda}}\right)+\mathrm{e}^{\sqrt{3}\sqrt[3]{\lambda}}+\mathrm{e}^{-\sqrt{3}\sqrt[3]{\lambda}}=6.
\end{equation*}
Using the Taylor expansion of $\mathrm{e}^{x}$ and $\cos x$, a straightforward calculation gives
\begin{equation*}
\left(2\sum_{n=0}^{\infty}\frac{\left(\frac{\sqrt{3}}{2} \sqrt[3]{\lambda}\right)^{2n}}{(2n)!}\right)\left(2\sum_{n=0}^{\infty}\frac{(-1)^n\left(\frac{3}{2} \sqrt[3]{\lambda}\right)^{2n}}{(2n)!}\right)+2\sum_{n=0}^{\infty}\frac{\left(\sqrt{3}\sqrt[3]{\lambda}\right)^{2n}}{(2n)!}=6.
\end{equation*}
From the binomial theorem, we obtain $\lambda=0$. As discussed in the case of (i) in (I), one gets  $y_{2}(1,\lambda)=0$, which contradicts $\displaystyle\lim_{\lambda\rightarrow 0}y_{2}(1,\lambda)=1$.

(II) For $\cos\gamma=0$. We compute $c_{3}=0$, then the coefficient matrix for $c_{1}$ and $c_{2}$ is
\begin{equation*}
M_{2}(\lambda):=\left(
\begin{array}{cc}
y_{1}^{\prime}(1,\lambda) & y_{2}^{\prime}(1,\lambda)-\mathrm{e}^{\mathrm{i}\phi}  \\
\cos\beta y_{1}(1,\lambda)-\mathrm{i}\sin\beta y_{1}^{\prime \prime}(1,\lambda) & \cos\beta y_{2}(1,\lambda)-\mathrm{i}\sin\beta y_{2}^{\prime \prime}(1,\lambda)%
\end{array}%
\right).
\end{equation*}
Similarly, according to (1) in Lemma \ref{entire}, a simple manipulation gives
\begin{equation*}
\left\{
\begin{array}{lr}
0=-\mathrm{i}\lambda y_{3}(1,\lambda), & \\
0=y_{1}(1,\lambda)-\mathrm{e}^{\mathrm{i}\phi}, & \\
0=\cos\beta y_{1}(1,\lambda)-\lambda\sin\beta  y_{2}(1,\lambda), & \\
0=\cos\beta y_{2}(1,\lambda)-\lambda \sin\beta y_{3}(1,\lambda). &
\end{array}
\right.
\end{equation*}
It yields
\begin{equation*}
-\mathrm{i}\lambda y_{3}(1,\lambda)=0.
\end{equation*}

(i) If $\lambda=0$, the non-trivial solution of equation $\mathrm{i}y^{\prime\prime\prime}(x)=\lambda y(x)$ is $y=ax^2+bx+c$, where $a, b, c \in \mathbb{C}$. Considering the boundary conditions (\ref{beta}) and $\cos\gamma=0$, we get a matrix of coefficients for $b$ and $c$, which is
\begin{equation*}
M_{3}(\lambda):=\left(
\begin{array}{cc}
1-\mathrm{e}^{\mathrm{i}\phi} & 0  \\
\cos\beta & \cos\beta %
\end{array}%
\right).
\end{equation*}
From the relation between the rank of the matrix and the multiplicity of eigenvalues, we obtain that $0$ is the eigenvalue of the operator $\widetilde{L}_{0}$ and has multiplicity $2$ if and only if $\cos\beta=0$, $\mathrm{e}^{\mathrm{i}\phi}-1=0$. If these two conditions are not simultaneously true, $0$ may be the simple eigenvalue of the operator $\widetilde{L}_{0}$.

(ii) If $y_{3}(1,\lambda)=0$, we obtain
\begin{equation*}
\cos\beta y_{2}(1,\lambda)=0.
\end{equation*}

\large{\textcircled{\small{i}}} If $y_{2}(1,\lambda)=0$, due to $y_{3}(1,\lambda)=0$, a straightforward calculation gives that $\lambda=0$, which contradicts $\displaystyle\lim_{\lambda\rightarrow 0}y_{2}(1,\lambda)=1$.

\large{\textcircled{\small{ii}}} If $\cos\beta=0$, then $\sin\beta\neq0$. One gets
\begin{equation*}
-\lambda y_{2}(1,\lambda)=0.
\end{equation*}
If $\lambda=0$, then the discussion is the same as (i) in (II), we get that $0$ is the eigenvalue of the operator $\widetilde{L}_{0}$ and has multiplicity $2$ if and only if $\cos\beta=0$, $\mathrm{e}^{\mathrm{i}\phi}-1=0$. 
If these two conditions are not simultaneously true, $0$ may be the simple eigenvalue of the operator $\widetilde{L}_{0}$. The case of $y_{2}(1,\lambda)=0$ has already been discussed in (ii) of (II).

In conclusion, for $\cos^2\gamma+\cos^2\beta+(\mathrm{e}^{\mathrm{i}\phi}-1)^2\neq0$, we have (1), i.e., the operator $\widetilde{L}_{0}$ only has simple eigenvalues. For $\cos^2\gamma+\cos^2\beta+(\mathrm{e}^{\mathrm{i}\phi}-1)^2=0$, one gets (2), that is, the operator $\widetilde{L}_{0}$ has simple eigenvalues, except for the eigenvalue $0$ which has multiplicity $2$.
\end{proof}

$\mathbf{Declarations}$\\

$\mathbf{Data}$  $\mathbf{Availability}$ This article has no associated data.

$\mathbf{Conflict}$ $\mathbf{of}$ $\mathbf{interest}$  The authors have no conflicts to disclose.\\

School of Science, Civil Aviation University of China, Tianjin, 300300, People’s Republic of China.\

E-mail address: $yx\_liu@cauc.edu.cn$ $15225635093@163.com$ 
\end{document}